\documentclass[12pt, reqno]{amsart}
\usepackage{}
\usepackage{enumerate}
\usepackage{mathrsfs}
\usepackage{amsfonts}
\usepackage{xcolor}
\usepackage{amsmath}
\usepackage{amsthm}
\usepackage{amssymb}
\usepackage{cases}
\usepackage{bm}
\usepackage{tikz}
\usepackage[pdfstartview=FitH, colorlinks=true,  linkcolor=red, citecolor=blue]{hyperref}
\numberwithin{equation}{section}
\newtheorem{thm}{Theorem}[section]

\newtheorem{lem}{Lemma}[section]

\newtheorem{cor}{Corollary}[section]
\newtheorem{rem}{Remark}[section]
\newtheorem{example}{Example}[section]
\newtheorem{prop}{Proposition}[section]
\newtheorem{defn}{Definition}[section]

\begin{document}
%\title {Eigenvalue estimates and diameter bounds with curvature condition on signed graphs}
%\title[Diameter bounds of graphs]{Diameter bounds of graphs via eigenvalues and curvature of associated sign structures}
\title{Curvature, diameter and signs of graphs}
\author{Wei Chen}
\address{\parbox[l]{\textwidth}{School of Mathematical Sciences,
University of Science and Technology of China, Hefei 230026, China}}
\email{mathchenwei@ustc.edu.cn}

\author{Shiping Liu}

\address{\parbox[l]{1\textwidth}{School of Mathematical Sciences,
University of Science and Technology of China, Hefei 230026, China}}
\email{spliu@ustc.edu.cn  }


\begin{abstract}
We prove a Li-Yau type eigenvalue-diameter estimate for signed graphs. That is, the nonzero eigenvalues of the Laplacian of a non-negatively curved signed graph are lower bounded by $1/D^2$ up to a constant, where $D$ stands for the diameter. This leads to several interesting applications, including a volume estimate for non-negatively curved signed graphs in terms of frustration index and diameter, and a two-sided Li-Yau estimate for triangle-free graphs. Our proof is built upon a combination of Chung-Lin-Yau type gradient estimate and a new trick involving strong nodal domain walks of signed graphs. We further discuss extensions of part of our results to nonlinear Laplacians on signed graphs.
%We improve the diameter-eigenvalue bounds of Aldous, Babai-Szegedy, Diaconis-Stroock, Chung-Yau and Chung-Lin-Yau via considering sign structures of a given graph. 
\end{abstract}

\maketitle

%%%%%%%%%%%%%%%%%%%%%%%%%%%%%%%%%%%%%%%%%%%%%%%%%%%%%%%%%%%%%%%%%%%%%%%%%%

\vskip 0.5cm
\section{Introduction}\label{Introduction}
\vskip 0.3cm
%%%%%%%%%%%%%%%%%%%%%%%%%%%%%%%%%%%%%%%%%%%%%%%%%%%%%%%%%%%%%%%%%%%%%%%%%%
There is a long history on the study of diameter bounds of graphs or finite Markov chains via their eigenvalues \cite{AM85, Chung89, Chung06, CFM94, CGY96, CGY97, HT01, LY10, LPS88, Mohar91, LP18}, see also the books \cite[Lemma 1.9 and Chapter 3]{Chung97} and \cite[Chapter 3]{Gri18} and references therein. Especially for highly symmetric $d$-regular graphs, it was established that 
\begin{equation}\label{eq:eigen_diameter}
    \lambda_2\geq \frac{C}{dD^2},
\end{equation}
where $\lambda_2$ is the first non-zero eigenvalue of the graph Laplacian, $D$ is the diameter of the graph and $C$ is an absolute constant. This was observed in the works of Aldous \cite[Lemma 3.6]{Aldous87} for Cayley graphs and  Babai-Szegedy \cite[Corollary 2.3]{BS91} for vertex-transitive graphs. Diaconis and Stroock \cite[Propositions 4 and 8]{DS91} pvoved for distance-transitive graphs that $\lambda_2\geq \frac{1}{D^2}$, that is, the dependence on vertex degree $d$ in (\ref{eq:eigen_diameter}) is not needed in this case.
The main technique involves estimating Cheeger constant in terms of the diameter and applying Cheeger inequality. 

This naturally reminds the Li-Yau inequality \cite{LY80} for the first nonzero eigenvalue of the Laplace-Beltrami operator on a compact Riemannian manifold with nonnegative Ricci curvature, which was further sharpened by Zhong-Yang \cite{ZY84}. To develop discrete analogues of their results, it is natural to consider proper discrete curvature notions. This was first carried out in the work of Chung-Yau \cite{ChungYau94, ChungYau96} on the so-called Ricci-flat graphs. A particular class of Ricci-flat graphs is consist of abelian Cayley graphs.
%The methods and results were motivated by the corresponding eigenvalue-diameter eistimate due to Li-Yau and sharpend by Zhong-Yang. 
The proof relies on establishing a gradient estimate/Harnack inequality for the eigenfunctions. This has been further extended in Chung-Lin-Yau \cite{Chung 14}, to graphs satisfying the Bakry-\'Emery curvature dimension inequalities $CD(K,N)$. Notice that a $d$-regular Ricci-flat graph satisfies the Bakry-\'Emery curvature dimension inequalities $CD(0,d)$, see e.g. \cite{Chung 14, M18}. Chung-Lin-Yau estimate has been further improved very recently by Meng and Lin \cite{MengLin24}.

In this article, our new observation is that this principle (\ref{eq:eigen_diameter}) of estimating nonzero eigenvalues of non-negatively curved graphs from below by $1/D^2$ also works for graphs associated with sign structures. A signed graph $(G,\sigma)$ is a graph $G=(V,E)$ associated with a sign $\sigma: E\to \{\pm 1\}$. The sign of a cycle in $(G,\sigma)$ is defined to be the product of all signs of edges in it. A signed graph $(G,\sigma)$ is called balanced if the sign of every cycle is $+1$. Spectral theory of signed graphs has led to a number of breakthroughs in various topics, including the solutions to the  sensitivity conjecture \cite{Huang19} and the open problems on equiangular lines \cite{JTYZZ}. There are close connections between the spectral theory of signed graphs and that of hypergraphs \cite{JMZ22} and simplicial complexes \cite{Gournay16}. 
In addition, signed  graphs have many other applications  in applied research topics such as modeling  biological networks, social situations, community detection, see, e.g. \cite{Harary, AtayLiu} and the references therein.

%Let $\lambda^\sigma$ be the first nonzero eigenvalue of the Laplacian $\Delta^\sigma$ of a signed graph $(G,\sigma)$. 
The Bakry-\'Emery curvature dimension inequalities $CD^\sigma(K,N)$ with $K\in \mathbb{R}$ and $N\in (0,\infty]$ on a signed graph $(G,\sigma)$ has been introduced by the second named author, M\"unch and Peyerimhoff in \cite{Liu 19}, extending the classical curvature dimension inequalities $CD(K,N)$ on graphs \cite{Elw91, LY10, Sch99}. Under this curvature condition, Buser type and Lichnerowicz type eigenvalue estimates were proved in \cite{Liu 19}. Furthermore, the explicit calculation of this curvature of signed graphs is reformulated in \cite{HL22} as solving the smallest eigenvalue of a so-called curvature matrix. 

Our first result is the following Chung-Lin-Yau type inequality for signed graphs.
\begin{thm}\label{thm:intro_1}
  Let $(G,\sigma)$ be a finite connected signed graph satisfying $CD^\sigma(K,\infty)$ and $\lambda^\sigma$ be the first nonzero eigenvalue of the Laplacian $\Delta^\sigma$ on $(G,\sigma)$. Then we have 
  \begin{equation}\label{eq:LiYau_Lich}
      \lambda^\sigma\geq \frac{C_1}{dD^2}+C_2K,
  \end{equation}
  where $d$ is the maximal vertex degree of $G$, $D$ is the diameter of $G$ and $C_1,C_2$ are absolute constants.
\end{thm}
This result will be restated in a more precise way in Theorem \ref{eigenvalue estimate}. Indeed, the smallest eigenvalue of $\Delta^\sigma$ vanishes if and only if $(G,\sigma)$ is balanced. That is, $\lambda^\sigma$ refers to the smallest eigenvalue of $\Delta^\sigma$ when $(G,\sigma)$ is unbalanced and to the second smallest eigenvalue when $(G,\sigma)$ is balanced. The estimate (\ref{eq:LiYau_Lich}) is a mixture of Li-Yau type and Lichnerowicz type eigenvalue estimates. It is first established by Chung, Lin and Yau \cite{Chung 14} for graphs with the all-positive sign. 

Theorem \ref{thm:intro_1} has several interesting consequences. Viewed as a diameter lower bound estimate, one can optimize the diameter lower bound over all possible choices of signs on the given graph $G$. This can possibly lead to better diameter estimates than only considering the all-positive sign. See Subsection \ref{section:diameter} for more detailed discussions. We point out that, in contrast, Bonnet-Myers type diameter upper bound estimate in terms of curvature cannot hold for signed graphs, as observed in \cite[Section 8]{HL22}. Indeed, there exists an infinite signed graph $(G,\sigma)$ satisfying $CD^\sigma(K,\infty)$ for $K>0$.

Another interesting result derived from Theorem \ref{thm:intro_1} is a volume estimate of non-negatively curved unbalanced signed graphs stated below.

\begin{thm}
    Let $(G,\sigma)$ be a finite connected unbalanced signed graph satisfying $CD^\sigma(0,\infty)$. Then the volume $\mathrm{vol}(G)$ of the graph $G$ satisfies
    \begin{equation}
        \mathrm{vol}(G)\leq C \iota^\sigma(G) dD,
    \end{equation}
    where $d$ is the maximal vertex degree of $G$, $D$ is the diameter of $G$, and $\iota^\sigma(G)$ is the frustration index of the signed graph $(G,\sigma)$.
\end{thm}
This result will be restated in a more precise way in Theorem \ref{thm:volume}.
Notice that the frustration index $\iota^\sigma(G)$ (see the definition (\ref{eq:frustration}) below) measures how far the signed graph $(G,\sigma)$ is from being balanced. And  $\iota^\sigma(G)=0$ if and only if $(G,\sigma)$ is balanced. This estimate brings new insight to the curvature condition $CD^\sigma(K,N)$ on signed graphs. Quite often this estimate can provide a criterion for checking whether a given signed graph satisfies $CD^\sigma(0,\infty)$ or not (See Example \ref{example:criterion}). Roughly speaking, for a signed graph $(G,\sigma)$ satisfying $CD^\sigma(0,\infty)$, if its volume is much larger than its diameter, then its frustration index should be large.

Theorem \ref{thm:intro_1} also shed new light on the eigenvalue estimates of the unsigned graph Laplacians. 
For a graph $G$ associated with the all-negative sign $\sigma_{-}$, it is known that the first nonzero eigenvalue of the corresponding Laplacian $\Delta^{\sigma_{-}}$ equals the first nonzero one of the sequence $2-\lambda_{|V|-i+1}$, $i=1,2,\ldots,|V|$ with $\lambda_i$'s being the eigenvalues of the graph Laplacian $\Delta$ ($\Delta^\sigma$ with the all-positive sign $\sigma$). 
Building upon this connection, we have the following result. 

\begin{thm}\label{thm:intr_3}
    Let $G$ be a finite connected triangle-free non-bipartite graph satisfying the classical curvature dimension inequalities $CD(0,\infty)$. Then the eigenvalues $0=\lambda_1<\lambda_2\leq\cdots\leq\lambda_{|V|}$ of the Laplacian $\Delta$ satisfy
    \begin{equation}
        \frac{c}{dD^2}\leq \lambda_2\leq \cdots\leq \lambda_{|V|}\leq 2-\frac{C}{dD^2},
    \end{equation}
    where $d$ is the maximal vertex degree of the graph $G$, $D$ is the diameter of $G$ and $c,C$ are absolute constants.
\end{thm}
This can be considered as a two-sided Li-Yau eigenvalue estimate. We comment that the assumption of non-bipartiteness is not essential.
Recall that a graph $G$ is bipartite if and only if it has no odd cycles, which is further equivalent to the property that $\lambda_{|V|}=2$. In fact, for bipartite graphs, the eigenvalues are symmetric with respect to $1$, that is, $\lambda_i=2-\lambda_{|V|-i+1}$ for any $i$. Therefore, for any bipartite graph $G$ satisfying $CD(0,\infty)$, we obtain the following two-sided Li-Yau estimate for free:
\begin{equation*}
      \frac{c}{dD^2}\leq \lambda_2\leq \cdots\leq \lambda_{|V|-1}\leq 2-\frac{c}{dD^2},
\end{equation*}
with $c$ being an absolute constant.

Our proofs are built upon the geometric analytic methods from Chung-Yau \cite{ChungYau94, ChungYau96} and Chung-Lin-Yau \cite{Chung 14}, that is, gradient estimate/Harnack inequality for eigenfunctions, and a new trick of applying gradient estimates along the strong nodal domain walks for the eigenfunctions on signed graphs. The concept of strong nodal domain walks were introduced in \cite{Ge 23}. 

Furthermore, we discuss possible extensions of our results to nonlinear $p$-Laplacians of signed graphs. Applying our trick involving strong nodal domain walks, we obtain an eigenvalue-diameter-volume estimate (see Theorem \ref{p-eigenvalue estimate}). We introduce curvature dimension inequalities $CD_p^\sigma(K,N)$ with $p>1$, $K\in \mathbb{R}$ and $N\in(0,\infty]$ involving the $p$-Laplacians on a signed graph $(G,\sigma)$ (see Definition \ref{def:CDpsigma]}). This concept is a generalization of various notions defined in \cite{Liu 19,Wang 21,Wang 24}. In terms of this new defined curvature notion, we prove a Lichnerowicz type eigenvalue estimate for $p$-Laplacian of a signed graph (see Theorem \ref{thm:Lich}). However, a proper gradient estimate for an eigenfunction of the $p$-Laplacian on a signed graph $(G,\sigma)$ under condition $CD_p^\sigma(K,N)$ is still missing. 

%%%%%%%%%%%%%%%%%%%%%%%%%%%%%%%%%%%%%%%%%%%%%%%%%%%%%%%%%%%%%%%%%%%%%%%%%%%%%
\vskip 0.5cm
\section{Preliminaries}

%%%%%%%%%%%%%%%%%%%%%%%%%%%%%%%%%%%%%%%%%%%%%%%%%%%%%%%%%%%%%%%%%%%%%%%%%%%%%%
We first prepare some notations on graphs. Let $G=(V,E)$ be a finite simple connected graph, where $V$ and $E$ stand for the sets of vertices and edges, respectively. We write $x\sim y$ if $\{x,y\}\in E$, and say in this case that $x$ is adjacent to $y$. We say the graph $G$ is locally finite if for any given $x\in V$, the set $\{y\in V: y\sim x\}$ is finite. We denote by $d_x:=|\{y\in V: y\sim x\}|$ the vertex degree, by $d:=\max_{x\in V}d_x$ the maximal vertex degree, and by $\mathrm{vol}(G):=\sum_{x\in V}d_x$ the volume of the graph $G$.
%Let $\omega:E\rightarrow\mathbb{R}^{+}, \{x,y\}\mapsto\omega_{xy}=\omega_{yx}$ be a symmetric weight and $\mu: V\rightarrow\mathbb{R}^{+},x\mapsto \mu_x$ be a vertex measure. 

The following sign structure turns out to be very helpful. We assign via a map $\sigma:E\rightarrow \{-1,1\}$ to each edge a sign. We write $\sigma_{xy}:=\sigma(\{x,y\})$ for short. We call $(G,\sigma)$ a signed graph. The sign of a cycle in $G$ is defined to be the product of all edge signs of the cycle. A signed graph $(G,\sigma)$ is called \emph{balanced} if the sign of every cycle in $G$ is $+1$ \cite{Harary}. A function $\tau: V\to \{+1,-1\}$ on the vertices is called a \emph{switching function}. For given $\sigma: E\to \{+1,-1\}$ and $\tau: V\to \{+1,-1\}$, the operation of switching $\sigma$ by $\tau$ refers to the operation of replacing $\sigma$ by 
the new map $\sigma^\tau: E\to \{+1,-1\}$ defined via
\[\sigma^\tau_{xy}:=\tau(x)\sigma_{xy}\tau(y),\,\,\text{for any}\,\,\{x,y\}\in E.\]
We say two signs $\sigma,\sigma': E\to\{+1,-1\}$ are \emph{switching equivalent} if there exists a switching function $\tau$ such that $\sigma'=\sigma^\tau$. 

The following characterization of balanced signed graphs via switching is due to Zaslavsky \cite[Corollary 3.3]{Zaslavsky}.
\begin{lem}[Zaslavsky's switching lemma]
A signed graph $(G,\sigma)$ is balanced if
and only if $\sigma$ is switching equivalent to the all $+1$ signature,
\end{lem}

On a signed graph $(G,\sigma)$, we consider the following signed Laplacian $\Delta^{\sigma}$:
\begin{equation}\label{Laplacian}
  \Delta^{\sigma}f(x):=\frac{1}{d_{x}}\sum_{y\sim x}(\sigma_{xy}f(y)-f(x)),\,\,\text{for any}\,\,f:V\to\mathbb{R}.
\end{equation}
If $\sigma_{xy}=+1$ for any edge $\{x,y\}$, the signed Laplacian coincides with the classical graph Laplacian, for which we denote simply by $\Delta$.

We say a function $f:V\rightarrow\mathbb{R}$ is an eigenfunction of $\Delta^\sigma$ corresponding to the eigenvalue $\lambda^\sigma$ if it holds for any $x\in V$ that
\begin{equation*}
  -\Delta^{\sigma}f(x)=\lambda^{\sigma}f(x).
\end{equation*}
We list the eigenvalues with multiplicity as below:
\[\lambda_1^\sigma\leq \lambda_2^\sigma\leq \cdots\leq \lambda_{|V|}^\sigma.\]
It is direct to check that $0\leq \lambda^\sigma_1\leq \lambda_{|V|}^\sigma\leq 2$ and for a switching function $\tau$ that
\[\Delta^{\sigma^\tau}f=\tau\Delta^\sigma (\tau f),\]
where $(\tau f)(x):=\tau(x)f(x)$. Therefore the eigenvalues of $\Delta^\sigma$ are invariant under switching operations of the sign. Indeed, the first eigenvalue $\lambda_1^\sigma=0$ if and only if the signed graph $(G,\sigma)$ is balanced (see e.g., \cite{AtayLiu}).

%\textbf{Recall the BE curvature from LMP}
Next, we recall the Bakry-\'Emery curvature dimension inequalities for signed graph introduced in \cite{Liu 19}. We first
recall the \emph{carr\'e du champ operators} on signed graphs. Indeed, for $f,g:V\rightarrow\mathbb{R}$, these operators are defined as
\begin{equation*}
  \Gamma^{\sigma}(f,g)(x):=\frac{1}{2}\left\{\Delta(f(x)g(x))-g(x)\Delta^{\sigma}f(x)-f(x)\Delta^{\sigma}g(x)\right\},
\end{equation*} 
and 
\begin{equation*}
  \Gamma_{2}^{\sigma}(f,g)(x):=\frac{1}{2}\left\{\Delta\Gamma^{\sigma}(f,g)(x)-\Gamma^{\sigma}(g,\Delta^{\sigma}f)(x)-\Gamma^{\sigma}(f,\Delta^{\sigma}g)(x)\right\}.
\end{equation*} 
We emphasize that utilization of both the Laplacian operator $\Delta$ and and the signed Laplacian operator $\Delta^{\sigma}$ in the above definitions is crucial.

The following lemma expressing $\Gamma^{\sigma}$ alternatively follows from a direct calculation.
\begin{lem} For any functions $f,g: V\to \mathbb{R}$, we have
  \begin{equation*}
      \Gamma^{\sigma}(f,g)(x)=\frac{1}{2d_{x}}\sum_{y\sim x}(\sigma_{xy} f(y)-f(x))(\sigma_{xy}g(y)-g(x)).
  \end{equation*}
  In particular, we obtain
  \begin{equation*}
     \Gamma^{\sigma}(f,f)(x)=\frac{1}{2d_{x}}\sum_{y\sim x}(\sigma_{xy}f(y)-f(x))^{2}:=\frac{1}{2}|\nabla^{\sigma}f|^{2}(x). 
  \end{equation*}
  \end{lem}

%Whenever the equation holds at every $x\in V$, we will omit the variable $x$ if there is no confusion. 
We write $\Gamma^{\sigma}(f)=\Gamma^{\sigma}(f,f)$ and $\Gamma_{2}^{\sigma}(f)=\Gamma_{2}^{\sigma}(f,f)$, for short. 
%Therefore we can give two important definition used in the following paper. We call curvature-dimension condition and strong nodal domain of signed graph.
\begin{defn}[{\cite[Definition 3.2]{Liu 19}}]
   Let $(G,\sigma)$ be a signed graph. We say $(G,\sigma)$ satisfies curvature-dimension inequality $CD^\sigma(K, N)$ for $N\in (0,\infty]$ at a vertex $x$ if it holds that
    \begin{equation*}
        \Gamma_{2}^{\sigma}(f)(x)\geq\frac{1}{N}(\Delta^{\sigma}f)^{2}(x)+K\Gamma^{\sigma}(f)(x),
    \end{equation*}
    for any function $f: V\to \mathbb{R}$. Here we use the convention $\frac{1}{N}=0$ in the case $N=\infty$.
    We say the signed graph $(G,\sigma)$ satisfies $CD^\sigma(K, N)$, if it satisfies $CD^\sigma(K, N)$ at every vertex.
\end{defn}
In case that $\sigma$ is the all-positive sign, $CD^\sigma(K,N)$ condition reduces to the classical Bakry-\'Emery curvature dimension inequality $CD(K,N)$ introduced on graphs in \cite{Elw91,LY10,Sch99}. 
Notice that the property of satisfying curvature dimension inequality $CD^\sigma(K,N)$ is invariant under switching operations of the sign \cite[Proposition 3.5]{Liu 19}. If the signs of $3$- and $4$-cycles of a signed graph $(G,\sigma)$ are all positive, it satisfies $CD^\sigma(K,N)$ if and only if the underlying graph $G$ satisfies the classical Bakry-\'Emery curvature dimension inequality $CD(K,N)$ \cite[Proposition 3.6]{Liu 19}. Any locally finite signed graph $(G,\sigma)$ satisfies \[CD^\sigma(-1+\frac{2}{d},2),\] where $d$ stands for the maximal vertex degree \cite[Corollary 3.8]{Liu 19}.
%As in \cite{Chung 14}, we also have following generalization inequality for signed version.
The following Lichnerowicz type eigenvalue estimate was proved in \cite{Liu 19}.
\begin{lem}[{\cite[Theorem 6.1]{Liu 19}}]\label{curvature estimate}
 Let $(G,\sigma)$ be a connected signed graph satisfying $CD^\sigma(K, N)$ with $K\in \mathbb{R}$ and $N\in (0,\infty]$. Then it holds for any non-zero eigenvalue $\lambda^{\sigma}$ of $\Delta^{\sigma}$ that
 \begin{equation*}
     \frac{N-1}{N}K\leq\lambda^{\sigma},
 \end{equation*}
 where we use the convention $\frac{N-1}{N}=1$ in the case $N=\infty$.
\end{lem}
%From now on, we will omit $\sigma$ of $K^{\sigma}$ and write simply $K$ whenever there is no confusion. 

To conclude this section, we recall the definition of the so-called strong nodal domains of a function on a signed graph, which has been introduced in \cite{M16, Ge 23}.
\begin{defn}[Strong nodal domains]
 Let $(G,\sigma)$ be a signed graph and $f:V\rightarrow\mathbb{R}$ be a function. A sequence $\{x_{k}\}_{k=1}^{n}$ of vertices is called a \emph{strong nodal domain walk} of $f$ if \[x_k\sim x_{k+1}\,\, \text{ and}\,\, \,f(x_{k})\sigma_{x_{k}x_{k+1}}f(x_{k+1})>0\] for each $k=1,2,...,n-1$. Let $\Omega=\{v\in V:f(v)\neq0\}$ be the set of non-zeros of $f$. We say two vertices $x,y\in \Omega$ are equivalent if and only if either  $x=y$ or there exists an strong nodal domain walk connecting $x$ and $y$. Let $\{S_{i}\}_{i=1}^{p}$ be the corresponding equivalence classes of $\Omega$. We call the induced subgraph of each $S_i$ a \emph{strong nodal domain} of the function $f$.
\end{defn}
%%%%%%%%%%%%%%%%%%%%%%%%%%%%%%%%%%%%%%%%%%%%%%%%%%%%%%%%%%%%%%%%%%%%%%%%%%%%%%%%%%%%%%%%%%%%%%%%%%%%%%%%%%%%%
\vskip 0.5cm
\section{Laplacian eigenvalues and diameter}
\subsection{Harnack inequality for eigenfunctions}
We prove the following Harnack inequality for eigenfunctions of $\Delta^\sigma$.
\begin{thm}\label{harnack}
  Suppose that a finite connected signed graph $(G,\sigma)$ satisfies $CD^{\sigma}(K,N)$ and $f:V\rightarrow \mathbb{R}$ is an eigenfunction of the signed Laplacian $\Delta^{\sigma}$ to a nonzero eigenvalue $\lambda^{\sigma}$. Then it holds for all $x\in V$ and $\alpha>2-2K/\lambda^{\sigma}$ that
  \begin{equation*}
    |\nabla^{\sigma}f|^{2}(x)+\alpha\lambda^{\sigma} f^{2}(x)\leq\frac{(\alpha^{2}-\frac{4}{N})\lambda^{\sigma}+2K\alpha}{(\alpha-2)\lambda^{\sigma}+2K}\lambda^{\sigma}\cdot \max_{z\in V}f^{2}(z).
  \end{equation*}
\end{thm}

\begin{proof}
  We first calculate the Laplacian of function $|\nabla^{\sigma}f|^{2}(x)+\alpha\lambda^{\sigma} f^{2}(x)$. Indeed, we have by the definition of $\Gamma_2^\sigma$ that
  \begin{equation*}
    \begin{split}
       (-\Delta)|\nabla^{\sigma}f|^{2}(x) 
         & = -2\Delta\Gamma^\sigma(f)(x)=-4\Gamma_2^\sigma(f)(x)-4\Gamma^\sigma(f,\Delta^\sigma f) \\
         &\leq -\frac{4}{N}(\Delta^\sigma f)^2(x)-4K\Gamma^\sigma(f)(x)-4\Gamma^\sigma(f,\Delta^\sigma f)\\
       &=(2\lambda^{\sigma}-2K)|\nabla^{\sigma}f|^{2}(x)-\frac{4}{N}(\lambda^{\sigma})^{2}f^{2}(x).
    \end{split}
  \end{equation*} 
  In the above we apply the curvature dimension inequality $CD^\sigma(K,N)$ and the equation $\Delta^\sigma f+\lambda^\sigma f=0$.
  Moreover we derive from the definition of $\Gamma^\sigma$ that
  \begin{equation*}
    \begin{split}
       (-\Delta)f^{2}(x)  =-2\Gamma^\sigma(f)(x)-2f(x)\Delta^\sigma f(x) =2\lambda^{\sigma}f^{2}(x)-|\nabla^{\sigma}f|^{2}(x).
    \end{split}
  \end{equation*}
  
  Putting the above two calculations together yields
  \begin{equation*}
    \begin{split}
       &(-\Delta)(|\nabla^{\sigma}f|^{2}+\alpha\lambda^{\sigma}f^{2})(x) \\
        \leq  & (2\lambda^{\sigma}-\alpha\lambda^{\sigma}-2K)|\nabla^{\sigma}f|^{2}(x)+\left(2\alpha-\frac{4}{N}\right)(\lambda^{\sigma})^{2}f^{2}(x).
    \end{split}
  \end{equation*}
  
Let $v$ be a vertex such that
  \begin{equation*}
    |\nabla^{\sigma}f|^{2}(v)+\alpha\lambda^{\sigma} f^{2}(v)=\max_{x\in V}\left\{|\nabla^{\sigma}f|^{2}(x)+\alpha\lambda^{\sigma} f^{2}(x)\right\}.
  \end{equation*}
 
Therefore, we derive at the vertex $v$ that
   \begin{equation*}
     \begin{split}
        0 & \leq (-\Delta)(|\nabla^{\sigma}f|^{2}+\alpha\lambda^{\sigma}f^{2})(v) \\
          & \leq (2\lambda^{\sigma}-\alpha\lambda^{\sigma}-2K)\cdot|\nabla^{\sigma}f|^{2}(v)+\left(2\alpha-\frac{4}{N}\right)(\lambda^{\sigma})^{2}f^{2}(v).
     \end{split}
   \end{equation*}
 
 For $\alpha>2-\frac{2K}{\lambda^{\sigma}}$, the above inequality implies
 \begin{equation*}
   |\nabla^{\sigma}f|^{2}(v)\leq\frac{2\alpha-\frac{4}{N}}{(\alpha-2)\lambda^{\sigma}+2K}(\lambda^{\sigma})^{2}f^{2}(v).
 \end{equation*}  
 
 Then for every $x\in V$, we derive
 \begin{equation*}
 \begin{split}
   |\nabla^{\sigma}f|^{2}(x)+\alpha\lambda^{\sigma}f^{2}(x)&\leq |\nabla^{\sigma}f|^{2}(v)+\alpha\lambda^{\sigma} f^{2}(v)\\
   &\leq\frac{(\alpha^{2}-\frac{4}{N})\lambda^{\sigma}+2K\alpha}{(\alpha-2)\lambda^{\sigma}+2K}\cdot\lambda^{\sigma}\cdot\max_{z\in V}f^{2}(z),
   \end{split}
 \end{equation*}
 which completes the proof.  
   \end{proof}

For any $\varepsilon>0$, we set $\alpha=2+\varepsilon-\frac{2K}{\lambda^{\sigma}}>0$. Notice that the positivity of $\alpha$ follows from Lemma \ref{curvature estimate}. Consequently, we obtain the following corollary.

\begin{cor}
    \label{gradient estimate}
  Suppose that a finite connected signed graph $G$ satisfies $CD^{\sigma}(K,N)$ and $f:V\rightarrow \mathbb{R}$ is an eigenfunction of the signed Laplacian $\Delta^{\sigma}$ to a nonzero eigenvalue $\lambda^{\sigma}$. Then it holds for all $x\in V$ and $\varepsilon>0$ that
  \begin{equation*}
    |\nabla^{\sigma}f|^{2}(x)\leq\left(\left((2+\varepsilon)^{2}-\frac{4}{N}\right)\frac{\lambda^{\sigma}}{\varepsilon}-\left(\frac{4}{\varepsilon}+2\right)K\right)\cdot \max_{z\in V}f^{2}(z).
  \end{equation*}
\end{cor}

\begin{rem}
  In case that $\sigma\equiv +1$ and $\varepsilon=2$, the above results reduces to \cite[Theorem 3.3 and Corollary 3.4]{Chung 14}.
\end{rem}

\subsection{Eigenvalue estimate}

In this subsection, we show the following eigenvalue estimate, which is achived by combining the Harnack inequality with tricks about strong nodal domains of eigenfunctions.

\begin{thm}\label{eigenvalue estimate}
  Let $\varepsilon>0$ be any positive number. Let $(G,\sigma)$ be a finite connected signed graph satisfying $CD^{\sigma}(K,N)$ with \begin{equation}\label{eq:Ncondition} N>\frac{4}{(2+\varepsilon)^2}\end{equation} and $\lambda^{\sigma}$ be a nonzero eigenvalue of the corresponding Laplacian $\Delta^\sigma$. 
  %If $\sigma$ is nonbalanced, $f$ has only one strong nodal domain and is non-zero for all $x\in V$,
  Then, we have
  \begin{equation}\label{estimate I}
     \lambda^{\sigma}\geq\frac{\varepsilon}{\left((2+\varepsilon)^{2}-\frac{4}{N}\right)}\cdot\frac{1}{d(D+1)\left\lceil (D+1)/2\right\rceil}+\frac{2(2+\varepsilon)}{\left((2+\varepsilon)^{2}-\frac{4}{N}\right)}K,
  \end{equation}
 % otherwise, we get
%  \begin{equation}\label{estimate II}
%     \lambda^{\sigma}\geq\frac{\varepsilon}{\left((2+\varepsilon)^{2}-\frac{4}{N}\right)}\cdot\frac{1}{dD\left\lceil D/2\right\rceil}+\frac{2(2+\varepsilon)}{\left((2+\varepsilon)^{2}-\frac{4}{N}\right)}K,
%  \end{equation}
  where $d$ is the maximal vertex degree, $D$ is the diameter of the graph $G$, and $\lceil\cdot\rceil$ denotes the smallest integer larger than the given number. When the multiplicity of $\lambda^\sigma$ is at least $2$ or the signed graph  $(G,\sigma)$ is balanced, the estimate can be improved as below:
  \begin{equation}\label{eq:estimatebetter}
      \lambda^{\sigma}\geq\frac{\varepsilon}{\left((2+\varepsilon)^{2}-\frac{4}{N}\right)}\cdot\frac{1}{dD\left\lceil D/2\right\rceil}+\frac{2(2+\varepsilon)}{\left((2+\varepsilon)^{2}-\frac{4}{N}\right)}K.
\end{equation}
\end{thm}

\begin{rem}
    Note that for $N\in [1,+\infty]$, the condition (\ref{eq:Ncondition}) holds for any positive number $\varepsilon>0$. It is proved in \cite[Corollary 6.2]{Liu 19} that if a signed graph $(G,\sigma)$ satisfies $CD^\sigma(K,N)$ with $N\in (0,1)$, then we have $K<0$.
\end{rem}

\begin{proof}
  Let $f$ be an eigenfunction to the eigenvalue $\lambda^\sigma$, i.e., $\Delta^\sigma f+\lambda^\sigma f=0$. By rescaling, we can assume that
    \begin{equation*}
    \max_{z\in V}f(z)=1\quad \mathrm{and} \quad\min_{z\in V}f(z)\geq-1.
  \end{equation*} 
 Let $x_{0}\in V$ be a vertex such that $f(x_{0})=\max_{z\in V}f(z)=1$. Let $x_{0}\sim x_{1}\sim ...\sim x_{t}$ be a path starting from the vertex $x_{0}$. By Corollary \ref{gradient estimate}, we obtain
  \begin{equation*}
  \begin{split}
    [f(x_{i-1})-\sigma_{x_{i-1}x_{i}} f(x_{i})]^{2}&+[f(x_{i})-\sigma_{x_{i}x_{i+1}} f(x_{i+1})]^{2}\leq d\cdot|\nabla^{\sigma}f|^{2}(x_{i})\\
    &\leq d\cdot\left(\left((2+\varepsilon)^{2}-\frac{4}{N}\right)\frac{\lambda^{\sigma}}{\varepsilon}-\left(\frac{4}{\varepsilon}+2\right)K\right).
   \end{split}
  \end{equation*}
Therefore, we derive via summing up the above inequalities that
  \begin{equation}\label{yi}
    \sum_{i=0}^{t-1}[f(x_{i})-\sigma_{x_{i}x_{i+1}} f(x_{i+1})]^{2}\leq d\left\lceil\frac{t}{2}\right\rceil\left(\left((2+\varepsilon)^{2}-\frac{4}{N}\right)\frac{\lambda^{\sigma}}{\varepsilon}-\left(\frac{4}{\varepsilon}+2\right)K\right).
  \end{equation}
% where $\left[\frac{D}{2}\right]$ denote to round up $\frac{D}{2}$.
 Next, we bound the left hand side of the above inequality from below:
 \begin{align}\label{er}
       \sum_{i=0}^{t-1}[f(x_{i})-\sigma_{x_{i}x_{i+1}} f(x_{i+1})]^{2} & =[f(x_{0})-\sigma_{x_{0}x_{1}} f(x_{1})]^{2}+\sigma_{x_{0}x_{1}}^{2}[f(x_{1})-\sigma_{x_{1}x_{2}} f(x_{2})]^{2}\notag\\
       &\quad+\cdots+(\sigma_{x_{0}x_{1}}\sigma_{x_{1}x_{2}}\cdots\sigma_{x_{i-1}x_{i}})^{2}[f(x_{i})-\sigma_{x_{i}x_{i+1}} f(x_{i+1})]^{2}\notag\\
       &\quad+\cdots+(\sigma_{x_{0}x_{1}}\sigma_{x_{1}x_{2}}\cdots\sigma_{x_{t-2}x_{t-1}})^{2}[f(x_{t-1})-\sigma_{x_{t-1}x_{t}} f(x_{t})]^{2}\notag\\
       &=\sum_{i=0}^{t-1}\left[\sigma(P_{x_0x_i})f(x_i)-\sigma(P_{x_0x_{i+1}})f(x_{i+1})\right]^2\notag\\
         & \geq\frac{1}{t}[f(x_{0})-\sigma(P_{x_{0}x_{t}}) f(x_{t})]^{2},
  \end{align}  
  where we use the notations $\sigma(P_{x_{0}x_{i}}):=\sigma_{x_{0}x_{1}}\sigma_{x_{1}x_{2}}\cdots\sigma_{x_{i-1}x_{i}}$ for each $i=1,\ldots, t$ and $\sigma(P_{x_0x_0}):=1$. The last inequality above is due to the Cauchy-Schwarz inequality.

  We have the following key observation: If the path $x_{0}\sim x_{1}\sim \cdots\sim x_{t}$ is a strong nodal domain walk of the function $f$, that is, $f(x_i)\sigma_{x_ix_{i+1}}f(x_{i+1})>0$ holds for each $i=0,1,\ldots,t-1$, then  we have 
  \begin{equation}\label{eq:strongnodal}
      \sigma(P_{x_0x_t})f(x_t)>0.
  \end{equation}
  We show this property inductively. For $t=0$, we check $\sigma(P_{x_0x_0})f(x_0)=f(x_0)=1>0$. Suppose that $\sigma(P_{x_0x_i})f(x_i)>0$. Since $f(x_i)\sigma_{x_ix_{i+1}}f(x_{i+1})>0$, the values $f(x_i)$ and $\sigma_{x_ix_{i+1}}f(x_{i+1})$ share the same sign. Therefore, we derive
  \[\sigma(P_{x_0x_{i+1}})f(x_{i+1})=\sigma(P_{x_0x_i})\sigma_{x_ix_{i+1}}f(x_{i+1})>0.\] 
  This proves (\ref{eq:strongnodal}).
  %$f(x_{0})>0$ and $f(x_{0})\sigma_{x_{0}x_{1}}f(x_{1})>0$, so $\sigma_{x_{0}x_{1}}f(x_{1})>0$. And $f(x_{1})\sigma_{x_{1}x_{2}}f(x_{2})>0$, we derive $\sigma_{x_{0}x_{1}}\sigma_{x_{1}x_{2}}f(x_{2})>0$. Then if $x_{0},x_{1},...,x_{i}$ is a strong nodal domain path, we have $\sigma(P_{x_{0}x_{i}})f(x_{i})>0$. 

  From the above argument, we further observe the following property. If the path $x_{0}\sim x_{1}\sim \cdots\sim x_{t-1}$ is a strong nodal domain walk of the function $f$, but $x_{t-1}\sim x_t$ is not a strong nodal domain walk, that is, $f(x_{t-1})\sigma_{x_{t-1}x_t}f(x_t)\leq 0$, then we have
  \begin{equation}\label{eq:nonstrongnodal}
      \sigma(P_{x_0x_t})f(x_t)\leq 0.
  \end{equation}

  We divide the proof of (\ref{estimate I}) into two cases.
\begin{itemize}
    \item [\textbf{Case 1.}] There exists a path $x_0\sim x_1\sim\cdots\sim x_t$ such that $x_0\sim x_1\sim\cdots\sim x_{t-1}$ is a strong nodal domain walk of the function $f$ but $x_{t-1}\sim x_t$ is not, and $t\leq D$.
\end{itemize}
Combining the inequalities (\ref{yi}), (\ref{er}) and (\ref{eq:nonstrongnodal}) and the fact that $f(x_0)=1$, we arrive at
\begin{equation}\label{path inequality}
\begin{split}
 & d\left\lceil\frac{t}{2}\right\rceil\left(\left((2+\varepsilon)^{2}-\frac{4}{N}\right)\frac{\lambda^{\sigma}}{\varepsilon}-\left(\frac{4}{\varepsilon}+2\right)K\right)\\
 \geq &\frac{1}{t}[f(x_{0})-\sigma(P_{x_{0}x_{t}}) f(x_{t})]^{2}\geq\frac{1}{t}.
\end{split}
\end{equation}
Noticing that $t\leq D$ and the condition (\ref{eq:Ncondition}), we obtain (\ref{eq:estimatebetter}), which 
%\begin{equation}\label{eq:estimatebetter}
%      \lambda^{\sigma}\geq\frac{\varepsilon}{\left((2+\varepsilon)^{2}-\frac{4}{N}\right)}\cdot\frac{1}{dD\left\lceil D/2\right\rceil}+\frac{2(2+\varepsilon)}{\left((2+\varepsilon)^{2}-\frac{4}{N}\right)}K.
%\end{equation}
implies the estimate (\ref{estimate I}).
\begin{itemize}
    \item [\textbf{Case 2.}] Any path $x_0\sim x_1\sim\cdots\sim x_t$ starting from $x_0$ with $t\leq D$ is a strong nodal domain walk of the function $f$.
\end{itemize}
Notice that we have $f(x)\neq 0$ for any $x\in V$ in this case, since otherwise, a shortest path connecting $x_0$ and a zero of $f$ is not a strong nodal domain walk. 

Furthermore, we claim that there exists an edge $\{x,y\}\in E$ which is not a strong nodal domain walk of $f$. In fact, if, otherwise, every edge is a strong nodal domain walk of $f$, i.e., $f(x)\sigma_{xy}f(y)>0$ for any $\{x,y\}\in E$, then the signed graph $(G,\sigma)$ is balanced. Actually, switching $\sigma$ by $\tau$ such that 
\[\tau(x)=\frac{f(x)}{|f(x)|},\,\,\text{for any }\,\,x\in V,\]
yields $\sigma^\tau\equiv +1$. We check directly that $\tau$ is an eigenfunction to the zero eigenvalue of $\Delta^\sigma$. Recalling that $f$ is an eigenfunction to a nonzero eigenvalue $\lambda^\sigma$, we derive 
\[0=\langle f,\tau\rangle=\sum_{x\in V}f(x)\tau(x)d_x=\sum_{x\in V}|f(x)|d_x,\]
which is a contradiction.

Let $\{x,y\}\in E$ be such an edge which is not a strong nodal domain walk. Then we must have $d(x_0,x)=d(x_0,y)=D$, since otherwise a shortest path connection $x_0$ and one of $x$ and $y$ continued by the edge $\{x,y\}$ violates the assumption of \textbf{Case 2}. Let $x_0\sim x_1\sim \cdots\sim x_{t-1}=x$ be a shortest path connecting $x_0$ and $x$. Then, we have $t-1=D$ and the path is a strong nodal domain walk by the assumption of \textbf{Case 2}. We further set $x_t=y$. Applying the inequalities (\ref{yi}), (\ref{er}) and (\ref{eq:nonstrongnodal}) and the fact that $f(x_0)=1$ yields the inequality (\ref{path inequality}). Inserting $t=D+1$ leads to the inequality (\ref{estimate I}). 

When the graph $(G,\sigma)$ is balanced, we see from the argument for \textbf{Case 2} that, only \textbf{Case 1} can happen for the eigenfunctions of the first nonzero eigenvalue $\lambda^\sigma$. Therefore, we have the better estimate (\ref{eq:estimatebetter}). 

When the multiplicity of $\lambda^\sigma$ is at least $2$, there always exists an eigenfunction $f$ to $\lambda$ which has zeros. That is, for such an eigenfunction, \textbf{Case 2} above can not happen. Therefore, the better estimate (\ref{eq:estimatebetter}) holds. Thus, we complete the proof.
\end{proof}
\begin{rem}
    Notice from the above proof that, except the extremal \textbf{Case 2}, we in fact have a slightly better estimate (\ref{eq:estimatebetter}). Ignoring the curvature restrictions, there does exist unbalanced signed graphs for which \textbf{Case 2} happens for any eigenfunction of the smallest eigenvalue $\lambda^\sigma$, as demonstrated in Example \ref{ex:extremalCase} below. However, the curvature of the signed graph in Example \ref{ex:extremalCase} is negative.  
\end{rem}
It is natural to ask the following interesting question: Does there exist an unbalanced signed graph $(G,\sigma)$ whose smallest eigenvalue $\lambda^\sigma$ is simple and the corresponding eigenfunction satisfies the assumption of \textbf{Case 2}?
\begin{example}\label{ex:extremalCase} 
     Consider the signed graph $(G,\sigma)$ depicted in Figure \ref{fig:2}.
\begin{figure}[!htp]
	\centering
	\tikzset{vertex/.style={circle, draw, fill=black!20, inner sep=0pt, minimum width=3pt}}
	%\tikzstyle{every node}=[circle, draw, fill=black!20, inner sep=0pt, minimum width=4pt]
	\begin{tikzpicture}[scale=1.0]
 	 \draw (-3,0) -- (-1,0) node[midway, below, black]{$+$}
		-- (1,0) node[midway, below, black]{$+$}
        --  (3,0) node[midway, below, black]{$+$}        
        -- (0,5.196) node[midway, right, black]{$-$} 
	    -- (-1,3.464) node[midway, left, black]{$+$} 
        -- (-2,1.732) node[midway, left, black]{$+$} 		
        -- (-3,0) node[midway, left, black]{$+$}; 
      \draw (-1,0) -- (-2,1.732) node[midway, right, black]{$+$};

		\node at (-3,0) [vertex, label={[label distance=0mm]270: \small $1$}, fill=black] {};
        \node at (-1,0) [vertex, label={[label distance=0mm]270: \small $2$} ,fill=black] {};
		\node at (1,0) [vertex, label={[label distance=0mm]270: \small $3$} ,fill=black] {};
		\node at (3,0) [vertex, label={[label distance=0mm]315: \small $4$} ,fill=black] {};
        \node at (0,5.196) [vertex, label={[label distance=1mm]90: \small $5$} ,fill=black] {};
        \node at (-1,3.464) [vertex, label={[label distance=1mm]135: \small $6$} ,fill=black] {};
        \node at (-2,1.732) [vertex, label={[label distance=1mm]135: \small $7$} ,fill=black] {};
		
	\end{tikzpicture}
	\caption{The signed circle graph.}\label{fig:2}	
	\label{fig:es}
\end{figure}
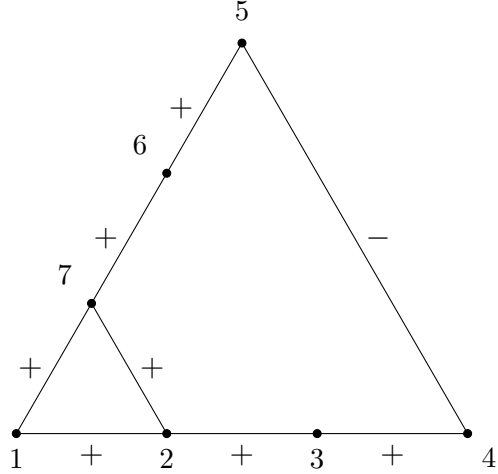
The matrix of the Lapalcian $-\Delta^{\sigma}$ is given below.
\begin{equation}\label{matrix}
\begin{pmatrix}  
   1&-1/2&0&0&0&0&-1/2\\
   -1/3&1&-1/3&0&0&0&-1/3\\
   0&-1/2&1&-1/2&0&0&0\\
   0&0&-1/2&1&1/2&0&0\\
   0&0&0&1/2&1&-1/2&0\\
   0&0&0&0&-1/2&1&-1/2\\
   -1/3&-1/3&0&0&0&-1/3&1\\
\end{pmatrix}
\end{equation}
One can check that the smallest eigenvalue of the signed graph $\lambda_{1}\approx 0.08$ is simple, and a corresponding (numerical) eigenfunction $f$ is 
\[\{1.087, 1, 0.672, 0.237, 0.237, 0.672, 1\}.\]
Clearly $f$ achieves its maximum at the vertex $x_0=1$ and any path starting from $x_0=1$ with length no larger than the diameter $3$ is a strong nodal domain walk. In fact, $d(1,4)=d(1,5)=D=3$ and $4\sim 5$ is not a strong nodal domain walk.

Notice that the signed graph $(G,\sigma)$ in Figure \ref{fig:2} satisfies $CD^\sigma(K,\infty)$ only for negative $K$. Since its only $3$-cycle is positive and it has no $4$-cycle, we derive from the switching invariance of the curvature dimension inequalities that $(G,\sigma)$ satisfies $CD^\sigma(K,N)$ if and only if the underlying graph $G$ satisfies $CD(K,N)$ \cite[Proposition 3.6]{Liu 19}. We check via the graph curvature calculator (https://www.mas.ncl.ac.uk/graph-curvature/, see \cite{CKL+22}) that $G$ satisfies $CD(-0.194, \infty)$.
\end{example}

Next, we present two consequences of Theorem \ref{eigenvalue estimate}. Indeed, Theorem \ref{eigenvalue estimate} can be viewed as a mixture of Lichnerowicz type and Li-Yau type eigenvalue estimates.
First, letting $\epsilon\to 0$ in Theorem \ref{eigenvalue estimate} yields the case $N>1$ of \cite[Theorem 6.1]{Liu 19} (see also Lemma \ref{curvature estimate}).
\begin{cor}{\cite{Liu 19}}
  Assume $(G,\sigma)$ satisfies $CD^{\sigma}(K,N)$ for $K>0$ and $N\in (1,\infty]$. Then we have for any non-zero eigenvalue $\lambda^{\sigma}$ of $\Delta^{\sigma}$ that
  \begin{equation*}
    \lambda^{\sigma}\geq\frac{N}{N-1}K,
  \end{equation*}
  where we use the convention $N/(N-1)=1$ in the case $N=\infty$.
\end{cor}
%\begin{proof}
 % The proof is only taking $\varepsilon\rightarrow0$ in inequalities (\ref{estimate I}) of Theorem \ref{eigenvalue estimate}.
%\end{proof}
Secondly, optimizing the coefficient in the case of $K=0$ leads to the following result.
\begin{cor}\label{cor:CD0N}
Assume $(G,\sigma)$ satisfies $CD^{\sigma}(0,N)$ for $N\in (1,\infty]$. Then we have for any non-zero eigenvalue $\lambda^{\sigma}$ of $\Delta^{\sigma}$ that
\[\lambda^\sigma\geq \frac{1}{4\left(1+\sqrt{(N-1)/N}\right)}\cdot\frac{1}{d(D+1)\left\lceil (D+1)/2\right\rceil}.\]
\end{cor}
\begin{proof}
This is derived from Theorem \ref{eigenvalue estimate} by observing that the coefficient \[\frac{\varepsilon}{(2+\varepsilon)^2-\frac{4}{N}},\,\,\varepsilon>0\]
attains its maximum at $\varepsilon=2\sqrt{(N-1)/N}$.
\end{proof}

\subsection{Viewed as Diameter bounds}\label{section:diameter}
In this subsection, we address the viewpoint of treating the eigenvalue estimate of Theorem \ref{eigenvalue estimate} as a diameter bound, which can be optimized via considering all possible signs on the underlying graph.
For convenience, we set $\varepsilon=2$ and $N=\infty$ in Theorem \ref{eigenvalue estimate}, which leads to the following consequence.

\begin{cor}\label{diameter bounds}
   Let $G$ be a finite connected graph. For any sign $\sigma$ on $G$ such that $CD^{\sigma}(K,\infty)$ holds, then the diameter $D$ of $G$ satisfies 
  \begin{equation}\label{eq:diameter1}
     (D+1) \left\lceil\frac{D+1}{2}\right\rceil\geq\frac{1}{4d(2\lambda^{\sigma}-K)},
  \end{equation}
  where $\lambda^\sigma$ is the first nonzero eigenvalue of $\Delta^\sigma$. Whenever $\lambda^\sigma$ is not simple or $(G,\sigma)$ is balanced, we have a better estimate 
\begin{equation}\label{eq:diameter2}
 D \left\lceil\frac{D}{2}\right\rceil\geq\frac{1}{4d(2\lambda^{\sigma}-K)}.
\end{equation}
\end{cor} 

We emphasize that, while the eigenvalue $\lambda^{\sigma}$ and curvature $K$ depend on the choice of the signs on the underlying graph, the diameter does not. Consequently, we can optimize the lower bounds on the diameter of a finite connected graph $G$ in Corollary \ref{diameter bounds} over all possible choices of signs on the graph $G$.

Next, we give examples to illustrate that the diameter bound can be improved by considering unbalanced signs. We first consider the $3$-cycle graph with the sign chosen as in Figure \ref{fig:1}.
\begin{figure}[!htp]
	\centering
	\tikzset{vertex/.style={circle, draw, fill=black!20, inner sep=0pt, minimum width=3pt}}
	%\tikzstyle{every node}=[circle, draw, fill=black!20, inner sep=0pt, minimum width=4pt]
	\begin{tikzpicture}[scale=1.0]
 	 \draw (-2,0) -- (2,0) node[midway, below, black]{$-$}
		-- (0,3.464) node[midway, right, black]{$+$}
        --  (-2,0) node[midway, left, black]{$+$}       ;

		\node at (-2,0) [vertex, label={[label distance=0mm]270: \small $1$}, fill=black] {};
           \node at (2,0) [vertex, label={[label distance=0mm]270: \small $3$} ,fill=black] {};
		\node at (0,3.464) [vertex, label={[label distance=0mm]90: \small $2$} ,fill=black] {};
		
	\end{tikzpicture}
	\caption{The complete signed graph $K_{3}^{\sigma}$.}\label{fig:1}	
	\label{fig:cg}
\end{figure}
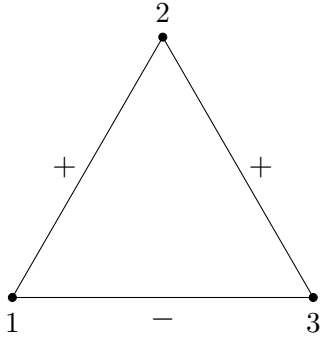
It is known, see, e.g. \cite[Section 3.5]{Liu 19}, the $3$-cycle graph satisfies $CD(\frac{5}{4},\infty)$ and  its first nonzero eigenvalue $\lambda=\frac{3}{2}$; Moreover, the signed $3$-cycle graph in Figure \ref{fig:1} satisfies $CD^\sigma(\frac{1}{4},\infty)$,  and its first eigenvalue $\lambda^{\sigma}=\frac{1}{2}$ has multiplicity $2$. 

If we choose the all-positive sign, we derive from the estimate (\ref{eq:diameter2}) for the $3$ cycle graph that $D\left\lceil D/2\right\rceil\geq 1/14$. If we choose instead, the sign given in Figure \ref{fig:1}, we obtain a better estimate from the estimate (\ref{eq:diameter2}) that $D\left\lceil D/2\right\rceil\geq 1/6$. 

In general, we consider any odd cycle graph $C_n$, i.e., a cycle graph with $n$ vertices where $n$ is odd. It is known \cite[Corollary 6.7]{Liu 19} that $(C_n,\sigma)$ with any sign $\sigma$ satisfies $CD^{\sigma}(0,N)$ for any $N\in [2,\infty]$. Moreover, the first nonzero eigenvalue $\lambda=1-\cos\frac{2\pi}{n}$ of $C_n$ has multiplicity $2$ and the first eigenvalue $\lambda^\sigma=1-\cos\frac{\pi}{n}$ of an unbalanced odd cycle $(C_n,\sigma)$  has also multiplicity $2$. Therefore, for an odd cycle $C_n$, we derive from (\ref{eq:diameter2}) using the balanced sign that 
\[ D \left\lceil\frac{D}{2}\right\rceil\geq\frac{1}{16(1-\cos\frac{2\pi}{n})}.\]
This can be improved by considering the unbalanced sign instead as
\[ D \left\lceil\frac{D}{2}\right\rceil\geq\frac{1}{16(1-\cos\frac{\pi}{n})}.\]

\begin{rem}
    The eigenvalue estimate in Theorem \ref{eigenvalue estimate} and hence the estimate $(\ref{eq:diameter2})$ for the case of balanced signed graphs are due to Chung-Lin-Yau \cite{Chung 14}. Their result is greatly improved very recently by Meng and Lin \cite{MengLin24}. Indeed, the result of Meng and Lin tells for a finite connected graph with all-positive sign satisfying $CD(K,\infty)$ that 
    \begin{equation}\label{eq:diameterML}
 D \left\lceil\frac{D}{2}\right\rceil\geq\frac{2(3-\sqrt{2})}{d(\lambda-(\sqrt{2}-1)K)}.
\end{equation}
It seems to us that the proof method of Meng and Lin does not work straightforwardly for unbalanced signed graphs. Their proof uses the fact $\Delta c=0$ for constant functions, which is no longer true for the Laplacian $\Delta^\sigma$ of unbalanced graphs.
\end{rem}

\subsection{Volume estimate for nonnegatively curved signed graphs}
In this subsection, we show a combination of Theorem \ref{eigenvalue estimate} with the Buser type inequalities established in \cite{Liu 19} yields an interesting volume estimate for unbalanced signed graphs with nonnegative curvature.

Recall that the frustration index $\iota^\sigma(G)$ of a signed graph $(G,\sigma)$ is defined as 
 \begin{equation}\label{eq:frustration}
     \iota^\sigma(G)=\min_{\tau: V\to \{ \pm 1\}}\sum_{\{x,y\}\in E}|\tau(x)-\sigma_{xy}\tau(y)|.
 \end{equation}
 It is direct to check that $\iota^\sigma(G)=0$ if and only if $(G,\sigma)$ is balanced. 
\begin{thm}\label{thm:volume}
    Let $(G,\sigma)$ be a finite connected unbalanced signed graph satisfying $CD^\sigma(0,N)$ for $N\in (1,\infty]$. Then we have 
    \begin{equation}\label{eq:volume}
        \mathrm{vol}(G)\leq 8\sqrt{(1+\sqrt{(N-1)/N})\ln 2}\cdot d\iota^\sigma(G)\sqrt{(D+1)\left\lceil(D+1)/2\right\rceil}.
    \end{equation}
\end{thm}
\begin{proof}
    Recall the Buser type inequalities for signed graph $(G,\sigma)$ satisfying $CD^\sigma(0,\infty)$ from \cite[Theorem 5.1]{Liu 19}. It tells that the first eigenvalue $\lambda^\sigma$ satisfies
    \begin{equation}\label{eq:Buser}
        \lambda^\sigma\leq 16(\ln 2)d\left(h^\sigma\right)^2,
    \end{equation}
    where $h^\sigma$ is the Cheeger constant for the signed graph $(G,\sigma)$ introduced in \cite{AtayLiu,LLPP}. Recall that 
    \[h^\sigma=\min_{\Omega\subseteq V}\frac{\iota^\sigma(G_\Omega)+|\partial \Omega|}{\mathrm{vol}(\Omega)},\]
    where $\mathrm{vol}(\Omega)=\sum_{x\in \Omega}d_x$, $|\partial\Omega|=|\{\{x,y\}\in E: x\in \Omega, y\not\in \Omega\}|$, $G_\Omega=(\Omega, E_\Omega)$ is the induced subgraph of $\Omega$, and $\iota^\sigma(G_\Omega)$ is the frustration index of $G_\Omega$ with the sign inherited from $(G,\sigma)$.
    
    Then the estimate (\ref{eq:volume}) follows from combining Corollary \ref{cor:CD0N} and (\ref{eq:Buser}) and the fact that 
    \[h^\sigma\leq \frac{\iota^\sigma(G)}{\mathrm{vol}(G)}.\]
\end{proof}

Recall the volume $\mathrm{vol}(G)$ of a graph $G$ equals twice the number of edges in $G$, and the frustration index of the corresponding signed graph $(G,\sigma)$ is counting twice the minimum number of negative edges over all signs that are switching equivalent to the given $\sigma$. Roughly speaking, Theorem \ref{thm:volume} tells that, if a signed graph $(G,\sigma)$ satisfies $CD^\sigma(0,\infty)$, then its frustration index should be large compared with its volume. In this sense, Theorem \ref{thm:volume} provides one criterion for checking whether a signed graph is nonnegatively curved or not.

\begin{example}\label{example:criterion}
    We consider hypercube graphs $Q^n$ of dimension $n\in \mathbb{N}$. Recall that $Q^n$ is the Cartesion product of $n$ copies of the complete graph $K_2$ with two vertices. We assign a sign $\sigma$ to $Q^n$ with precisely one negative edge. We hope to check whether $(Q^n,\sigma)$ satisfies $CD^\sigma(0,\infty)$ or not. 
    
    We recall that the hypercube $Q^n$ with the all-positive sign satisfies $CD(\frac{2}{n},\infty)$ (see, e.g., \cite[Example 7.15]{CLP20}). When $n=1$, the hypercube graph $Q^1=K_2$ has no cycle and hence $(Q^1,\sigma)$ is balanced and satisfies $CD^\sigma(2,\infty)$.
    
    Clearly, $(Q^n,\sigma),\,n\geq 2$ is unbalanced and the frustration index $\iota^\sigma(Q^n)$ equals $2$. Suppose that $(Q^n,\sigma),\,n\geq 2$ satisfies $CD^\sigma(0,\infty)$. Then we derive from Theorem \ref{thm:volume} that
    \[ 2^{2n}\leq 512\ln 2\cdot (n+1)\left\lceil\frac{n+1}{2}\right\rceil.\]
    This is a contradiction for $n\geq 7$.
    Therefore, we conclude that $(Q^n,\sigma),\,n\geq 7$ does not satisfy $CD^\sigma(0,\infty)$. In fact, one can check directly by the curvature matrix method for calculating $CD^\sigma(K,N)$ developed in \cite{HL22} that $(Q^2,\sigma)$ satisfies $CD^\sigma(0,\infty)$, and  $(Q^n,\sigma),\,n\geq 3$ does not satisfy $CD^\sigma(0,\infty)$. We present the details of the proof in Appendix \ref{section:Appendix}.

    For comparison, we consider the Cartesian product of $(Q^2,\sigma)$ with $Q^{n-2}$ associated the all-positive sign for $n\geq 2$. The underlying graph of this Cartesian product is $Q^n$, and the resulting sign $\sigma^\times$ on $Q^n$ produces $2^{n-2}$ negative edges. By the results about curvature of Cartesian products \cite[Appendix A.1]{Liu 19}, we conclude that $CD^{\sigma^\times}(0,\infty)$ holds for $(Q^n,\sigma^\times)$, $n\geq 2$.
\end{example}
\subsection{Graphs with the all-negative signs} In this subsection, we discuss another application of Theorem \ref{eigenvalue estimate}. Let $G=(V,E)$ be a finite connected graph. We assign a particular sign $\sigma_{-}: E\to \{\pm 1\}$ such that $\sigma_{-}\equiv -1$. For this signed graph $(G,\sigma_{-})$, the corresponding Laplacian $\Delta^{\sigma_{-}}$ has the following particular formulation:
\begin{equation*}
    \Delta^{\sigma_{-}}f(x):=-\frac{1}{d_x}\sum_{y\sim x}(f(y)+f(x)),\,\,\text{for any}\,\,f:V\to \mathbb{R}.
\end{equation*}
Let us denote by $0=\lambda_1<\lambda_2\leq\cdots\leq \lambda_{|V|}\leq 2$ the eigenvalues of the graph Laplacian $\Delta$ (which is the Laplaican $\Delta^\sigma$ with $\sigma$ being the all-positive sign). It is direct to check that the eigenvalues $0\leq \lambda_1^{\sigma_{-}}<\lambda_2^{\sigma_{-}}\leq\cdots\leq \lambda_{|V|}^{\sigma_{-}}\leq 2$ of the Laplacian $\Delta^{\sigma_{-}}$ satisfy
\begin{equation}\label{eq:bipartite}
    \lambda_i^{\sigma_{-}}=2-\lambda_{|V|-i+1},\,\,\text{for every}\,\,i=1,2,\ldots, |V|.
\end{equation}
Recall that 
\[2-\lambda_{|V|}=0\,\,\text{if and only if the graph $G$ is bipartite}.\]  A graph $G=(V,E)$ is bipartite if there exists a partition $V=V_1\sqcup V_2$ such that edges only connect vertices from different subsets $V_1, V_2$. So the first nonzero eigenvalue $\lambda^{\sigma_{-}}$ of $\Delta^{\sigma_{-}}$ coincides with $2-\lambda_{|V|-1}$ if $G$ is bipartite and with $2-\lambda_{|V|}$ if $G$ is non-bipartite. Indeed, for bipartite graphs, the eigenvalues are symmetric with respect to $1$, in particular, we have  $2-\lambda_{|V|-1}=\lambda_2$. Moreover,  $CD^{\sigma_{-}}(K,N)$ holds if and only if $CD(K,N)$ holds, since the all-positive sign and all-negative sign on a bipartite graph are switching equivalent. So we do not gain new insights to consider the all-negative sign $\sigma_{-}$ on a bipartite graph. For non-bipartite graphs, we have the following observation as a direct consequence of Theorem \ref{eigenvalue estimate}.

\begin{cor}\label{cor:allnegative}
    Let $G=(V,E)$ be a finite connected non-bipartite graph such that $(G,\sigma_{-})$ satisfies $CD^{\sigma_{-}}(0,N)$ for $N\in (1,\infty]$. Then we have 
    \begin{equation*}
        2-\lambda_{|V|}\geq \frac{1}{4\left(1+\sqrt{(N-1)/N}\right)}\cdot\frac{1}{d(D+1)\left\lceil (D+1)/2\right\rceil}.
    \end{equation*}
\end{cor}
To conclude, we prove Theorem \ref{thm:intr_3} stated in the Introduction.
\begin{proof}[Proof of Theorem \ref{thm:intr_3}]
    Recall from \cite[Proposition 3.6]{Liu 19} that a signed graph $(G,\sigma)$ with the sign of every $3$- and $4$-cycle being $+1$ satisfies $CD^\sigma(K,N)$ if and only if the underlying graph $G$ satisfies $CD(K,N)$. For a graph $G$ with the all-negative sign $\sigma_{-}$, the sign of every $4$-cycle is $+1$. Therefore, for a triangle-free graph $G$ associated with the sign $\sigma_{-}$, it satisfies $CD^{\sigma_{-}}(K,N)$ if and only the graph $G$ satisfies $CD(K,N)$. Therefore, the theorem follows directly from Theorem \ref{eigenvalue estimate} and Corollary \ref{cor:allnegative}, and the relations (\ref{eq:bipartite}).
\end{proof}

%%%%%%%%%%%%%%%%%%%%%%%%%%%%%%%%%%%%%%%%%%%%%%%%%%%%%%%%%%%%%%%%%%%%%%%%%%%%%%%%%%%%%%%%%%%%%%%%%%%%%%%%%%%%%%%%%%%%%%%%%%%%%%%%%
\vskip 0.5cm
\section{Eigenvalue estimates for p-Laplacians on signed graphs}

In this section, we discuss lower bound estimates for the first nonzero eigenvalue of $p$-Laplacian on signed graphs. For studies on $p$-Laplacians on signed graphs, we refer to \cite{JMZ22,GLZ23}. Using similar tricks of strong nodal domain paths as in the proof of Theorem \ref{eigenvalue estimate}, we derive an lower bound eigenvalue estimate in terms of diameter and volume of the underlying graph. Via introducing a proper notion of curvature dimension inequalities for the $p$-Laplacian on signed graphs, in analogue with the linear case developed in \cite{Liu 19}, we further show a Lichnerowicz type  estimates. Those results extend the previous ones of Wang and Huang \cite{Wang 21} on unsigned graphs.

For the readers' convenience, we first recall some preliminaries on the $p$-Laplacian of signed graphs.
For $p>1$, the signed $p$-Laplacian $\Delta_{p}^{\sigma}$ on a signed graph $(G, \sigma)$ is defined as
\begin{equation}\label{p-Laplacian}
  \Delta_{p}^{\sigma}f(x):=\frac{1}{d_{x}}\sum_{y\sim x}|\sigma_{xy}f(y)-f(x)|^{p-2}(\sigma_{xy}f(y)-f(x)),
\end{equation}
for any $f: V\to \mathbb{R}$ and any $x\in V$. Here, we use the convention that $|t|^{p-2}t=0$ when $t=0$.

If a nonzero function $f:V\rightarrow\mathbb{R}$ and a real number $\lambda_p^\sigma$ satisfies
\begin{equation*}
  -\Delta_{p}^{\sigma}f(x)=\lambda_{p}^{\sigma}|f(x)|^{p-2}f(x),\,\,\text{for any}\,\,x\in V, 
\end{equation*}
then we call $f$ an eigenfunction of $\Delta_{p}^{\sigma}$ on $(G,\sigma)$ to the eigenvalue $\lambda_{p}^{\sigma}$. In fact, the eigenvalues and eigenfunctions of $\Delta_p^\sigma$ are precisely the critical values and critical points of the following $p$-Rayleigh quotient: 
\begin{equation*}
  \mathcal{R}_{p}^{\sigma}(f):=\frac{\sum_{\{x,y\}\in E}|f(x)-\sigma_{xy}f(y)|^{p}}{\sum_{x\in V}|f(x)|^{p}d_{x}}, \,\,\text{for any}\,\,f:V\to\mathbb{R}.
\end{equation*}
In particular, all eigenvalues of $\Delta_p^\sigma$ are nonnegative. Moreover, the set of eigenvalues of $\Delta_p^\sigma$ is invariant under taking switching operations \cite[Proposition 2.5]{GLZ23}.

We focus on the first (smallest) nonzero eigenvalue $\lambda_p^\sigma$ of $\Delta_{p}^{\sigma}$ on a connected signed graph $(G,\sigma)$. If $(G,\sigma)$ is unbalanced, then the eigenvalues of its Laplacian $\Delta^\sigma_p$ are all positive \cite[Proposition 4.1]{GLZ23}. Hence, we have
\begin{equation*}
  \lambda_{p}^{\sigma}=\inf_{f\not \equiv0} \mathcal{R}_{p}^{\sigma}(f).
\end{equation*}
If, otherwise, $(G,\sigma)$ is balanced, then the smallest eigenvalue is $0$ and the switching function $\tau: V\to \{\pm 1\}$ such that $\sigma^\tau\equiv +1$ is a corresponding eigenfunction. The smallest nonzero eigenvalue $\lambda_p^\sigma$ satisfies (See e.g. \cite[Proposition 2.2]{DM19} and the references therein)
\begin{equation*}
\lambda_{p}^{\sigma}=\inf_{\substack{f\not\equiv 0,\\\sum_{x\in V}|f(x)|^{p-2}f(x)\tau(x) d_{x}=0}}\mathcal{R}_{p}^{\sigma}(f).
\end{equation*}

The following monotonicity properties of the first nonzero eigenvalue is a direct consequence of the monotonicity of \emph{variational eigenvalues}, which was first established in \cite[Theorem 1.1]{Zhang} for graph $p$-Laplacians and later extended to the signed graph setting in \cite[Theorem 1]{Ge 2023}.
\begin{thm}\label{monotonocity}
  Let $(G,\sigma)$ be a signed graph. For $1<p\leq q$, we have
  \begin{equation}\label{monotonicity i}
    2^{-p}\lambda_{p}^{\sigma}\geq 2^{-q}\lambda_{q}^{\sigma},
  \end{equation}
  and
  \begin{equation}\label{monotoniciy ii}
    p(2\lambda_{p}^{\sigma})^{\frac{1}{p}}\leq q(2\lambda_{q}^{\sigma})^{\frac{1}{q}}.
  \end{equation}
\end{thm}
\begin{proof}
    The variational eigenvalues $\lambda^\sigma_{p,1}\leq \lambda^\sigma_{p,2}\leq \cdots\leq \lambda^\sigma_{p,|V|}$ are defined using Lusternik–Schnirelman theory. We refer to \cite{Ge 2023} and the references therein for the precise definition. The result \cite[Theorem 1]{Ge 2023} tells that for any given $1\leq k\leq |V|$, the variational eigenvalue $\lambda^\sigma_{p,k}$ satisfies the monotonicity properties (\ref{monotonicity i}) and (\ref{monotoniciy ii}) with respective to $p$. Then the theorem follows directly from the fact that the first nonzero eigenvalue $\lambda^\sigma_p$ coincides with $\lambda^\sigma_{p,1}$ when $(G,\sigma)$ is unbalanced, and coincides with $\lambda^\sigma_{p,2}$ when $(G,\sigma)$ is balanced.
\end{proof}

\subsection{An eigenvalue lower bound via diameter and volume} In this subsection, we show the following lower bound estimate for the first nonzero eigenvalue of the $p$-Laplacian on a signed graph.
\begin{thm}\label{p-eigenvalue estimate}
  Let $(G,\sigma)$ be a finite connected signed graph. Denote by $D$ the diameter and by $\mathrm{vol}(G)$ the volume of the underlying graph $G$. Then the first nonzero eigenvalue $\lambda_{p}^{\sigma}$ of the $p$-Laplacian $\Delta^{\sigma}_p$ for any $p>1$ satisfies
  \begin{equation}\label{estimate 1}
    \lambda_{p}^{\sigma}\geq\frac{1}{(D+1)^{p-1}\mathrm{vol}(G)}.
  \end{equation}
\end{thm}
In case that $(G,\sigma)$ is balanced, this estimate is well-known for linear graph Laplacian ($p=2$), see \cite[Lemma 1.9]{Chung97} and is extended to graph $p$-Laplacians for any $p>1$ recently by Wang and Huang \cite{Wang 21}.
\begin{proof}
  Let $f$ be an eigenfunction of $\Delta^\sigma_P$ to the first nonzero eigenvalue $\lambda_p^\sigma$. By rescaling, we can assume that
  \begin{equation*}
    \max_{z\in V}f(z)=1\quad \mathrm{and} \quad\min_{z\in V}f(z)\geq-1.
  \end{equation*} 
  Let $x_{0}\in V$ be a vertex such that $f(x_{0})=\max_{z\in V}f(z)=1$. Let $x_{0}\sim x_{1}\sim ...\sim x_{t}$ be a path starting from the vertex $x_{0}$. By definition, we have
  \begin{equation}\label{rayleigh}
    \lambda_{p}^{\sigma}=\frac{\sum_{\{x,y\}\in E}|\sigma_{xy}f(y)-f(x)|^{p}}{\sum_{x}|f(x)|^{p}d_{x}}\geq\frac{\sum_{i=0}^{t-1}|f(x_{i})-\sigma_{x_{i}x_{i+1}}f(x_{i+1})|^{p}}{\mathrm{vol}(G)\cdot |f(x_{0})|^{p}}.
  \end{equation}
Then, we have 
\begin{align}\label{inequality}
       \sum_{i=0}^{t-1}|f(x_{i})-\sigma_{x_{i}x_{i+1}} f(x_{i+1})|^{p} & =|f(x_{0})-\sigma_{x_{0}x_{1}} f(x_{1})|^{p}+|\sigma_{x_{0}x_{1}}|^{p}|f(x_{i})-\sigma_{x_{i}x_{i+1}} f(x_{i+1})|^{p}\notag\\
       &\quad+\cdots+|\sigma_{x_{0}x_{1}}\sigma_{x_1x_2}\cdots\sigma_{x_{i-1}x_{i}}|^{p}|f(x_{i})-\sigma_{x_{i}x_{i+1}} f(x_{i+1})|^{p}\notag\\
       &\quad+\cdots+|\sigma_{x_{0}x_{1}}\sigma_{x_1x_2}\cdots\sigma_{x_{t-2}x_{t-1}}|^{p}|f(x_{t-1})-\sigma_{x_{t-1}x_{t}} f(x_{t})|^{p}\notag\\
         & \geq\frac{1}{t^{p-1}}|f(x_{0})-\sigma(P_{x_{0}x_{t}}) f(x_{t})|^{p},
\end{align}  
  where the last inequality follows from the inequality
  \begin{equation*}\label{sim CS}
    n^{p-1}(|a_{1}|^{p}+|a_{2}|^{p}+...+|a_{n}|^{p})\geq|a_{1}+a_{2}+...+a_{n}|^{p},\,\,\text{for}\,\,p>1,\,a_i\in \mathbb{R},\,i=1,2,\ldots,n.
  \end{equation*}
 By a similar argument as in the proof of Theorem \ref{eigenvalue estimate}, there exists a path $x_0\sim x_1\sim\cdots\sim x_t$ with $t\leq D+1$, such that $x_0\sim x_1\sim\cdots\sim x_{t-1}$ is a strong nodal domain walk of the function $f$ but $x_{t-1}\sim x_t$ is not. So we have $\sigma(P_{x_{0}x_{t}}) f(x_{t})\leq 0$ and hence
 \begin{equation}\label{equation:path}
     \frac{1}{t^{p-1}}|f(x_{0})-\sigma(P_{x_{0}x_{t}}) f(x_{t})|^{p}\geq\frac{1}{t^{p-1}}\geq\frac{1}{(D+1)^{p-1}}.
     \end{equation}
 Combining inequalities (\ref{rayleigh}), (\ref{inequality}) and (\ref{equation:path}), the proof is completed.  
  \end{proof}

The following corollary is a direct consequence of Theorem \ref{p-eigenvalue estimate} and the monotonicity property in Theorem \ref{monotonocity}.
\begin{thm}
  Let $(G,\sigma)$ be a finite connected signed graph. Then we have  for the first nonzero eigenvalue $\lambda_p^\sigma$ that
  \begin{equation*}
    \lambda_{p}^{\sigma}\geq\frac{1}{2}\left(\frac{2}{p\cdot \mathrm{vol}(G)}\right)^{p}
  \end{equation*} 
\end{thm}

\begin{proof}
  For $p\geq q>1$, we have
  \begin{equation*}
    p(2\lambda_{p}^{\sigma})^{\frac{1}{p}}\geq q(2\lambda_{q}^{\sigma})^{\frac{1}{q}}
    \geq q\left(\frac{2}{(D+1)^{q-1}\mathrm{vol}(G)}\right)^{\frac{1}{q}}.
  \end{equation*}
  Letting $q\to 1$, we complete the proof. 
\end{proof}

%As Corollary \ref{diameter bounds}, we get the following diameter bounds of signed $p$-Laplacian.
%\begin{cor}
 % Suppose a finite signed graph $G$ with diameter $D$. then we have
  %\begin{equation*}
   %D+1\geq\left(\frac{1}{\lambda_{p,1}^{\sigma}\mathrm{vol}(G)}\right)^{\frac{1}{p-1}}.
  %\end{equation*}
%\end{cor}

\subsection{A Lichnerowicz type eigenvalue estimate} In this subsection, we present a Lichnerowicz type estimate for the first nonzero eigenvalue of the $p$-Laplacian on a signed graph. For that purpose, we first introduce a notion of curvature dimension inequalities involving $p$-Laplacians of signed graphs. 

We start by deriving the linearization of the $p$-Laplacian $\Delta^\sigma_p$ at point $f$ on the signed graph $(G,\sigma)$. By a straightforward calculation, we have for any $\varphi: V\to \mathbb{R}$ that
\begin{align*}
     &\frac{\partial}{\partial\varepsilon}\bigg|_{\varepsilon=0}\Delta_{p}^{\sigma}(f+\varepsilon\varphi)(x)\\ =&\frac{\partial}{\partial\varepsilon}\bigg|_{\varepsilon=0}\frac{1}{d_{x}}\sum_{y\sim x}|\sigma_{xy}(f(y)+\varepsilon\varphi(y))-f(x)+\varepsilon\varphi(x)|^{p-2}(\sigma_{xy}(f(y)+\varepsilon\varphi(y))-f(x)+\varepsilon\varphi(x))\\
       =&(p-1)\frac{1}{d_{x}}\sum_{y\sim x}|\sigma_{xy}f(y)-f(x)|^{p-2}(\sigma_{xy}\varphi(y)-\varphi(x)).
\end{align*}
Note that the above calculation only works at $x\in V$ such that $\sigma_{xy}f(y)-f(x)\neq 0$ for each $y\sim x$ when $1<p<2$.

Therefore, we define the linearized signed $p$-Laplacian $ L_{p,f}^{\sigma}$ on signed graph at point $f$ by
\begin{equation}\label{linearized operator}
  L_{p,f}^{\sigma}\varphi(x):=\frac{1}{d_{x}}\sum_{y\sim x}|\sigma_{xy}f(y)-f(x)|^{p-2}(\sigma_{xy}\varphi(y)-\varphi(x)),
\end{equation}
and a modified linearized $p$-Laplacian $\mathscr{L}^\sigma_{p,f}$ at point $f$ by
\begin{equation}\label{eq:modified linearized operator}
\mathscr{L}^\sigma_{p,f}\varphi(x):=\frac{1}{d_{x}}\sum_{y\sim x}|\sigma_{xy}f(y)-f(x)|^{p-2}(\varphi(y)-\varphi(x)).
\end{equation}
In case $1<p<2$, they are well-defined only at $x\in V$ satisfying $\sigma_{xy}f(y)-f(x)\neq 0$ for each $y\sim x$.

Then, we introduce Bakry-\'Emery type curvature dimension inequalities for the $p$-Laplacian on a signed graph. 
\begin{defn}\label{def:CDpsigma]}
Let $(G,\sigma)$ be a signed graph and $f,g: V\to \mathbb{R}$ be two functions on it. 
    For any $p>1$ and any $x\in V$, we define
\begin{equation}
    \Gamma_p^\sigma(f,g)(x):=\frac{1}{2d_{x}}\sum_{y\sim x}|\sigma_{xy}f(y)-f(x)|^{p-2}(\sigma_{xy}f(y)-f(x))(\sigma_{xy}g(y)-g(x)).
\end{equation}
For $1<p<2$ and $x\in V$ such that $\sigma_{xy}f(y)-f(x)\neq 0$ holds at each neighbor $y\sim x$, or for $p\geq 2$ and any $x\in V$, we define
\begin{equation}
    \Gamma_{p,2}^\sigma(f,f)(x):=\frac{1}{2}\mathscr{L}_{p,f}^\sigma(\Gamma^\sigma_p(f,f))(x)-\Gamma_p^\sigma(f,\Delta_p^\sigma f)(x).
\end{equation}
\end{defn}

%Notice that, in the above definitions we use the convention $|t|^{p-2}t=0$ at $t=0$. Moreover, in the case $1<p<2$, we have at $x\in V$ with $\sigma_{xy}f(y)-f(x)=0$ for each $y\sim x$ that $\Gamma_p^\sigma(f,g)(x)=0$ and hence
%\begin{equation*}
    % \Gamma_{p,2}^\sigma(f,f)(x)=\frac{1}{2d_x}\sum_{y\sim x}|\sigma_{xy}f(y)-f(x)|^{p-2}\Gamma_p^\sigma(f,f)(y).
    % \end{equation*}
    %Here we use the convention that
%\begin{equation*}
%|\sigma_{xy}f(y)-f(x)|^{p-2}\Gamma_p^\sigma(f,f)(y)=\begin{cases}
%0 & \text{if $\Gamma_p^\sigma(f,f)(y)=0$,} \\
%+\infty & \text{if otherwise $\Gamma_p^\sigma(f,f)(y)>0$.}
%\end{cases}
%\end{equation*}
\begin{rem}
Let $f,g:V\to \mathbb{R}$ be two functions.
    When $p\geq 2$ at any $x\in V$, or when $1<p<2$ at $x\in V$ with $\sigma_{xy}f(y)-f(x)\neq 0$ for each $y\sim x$, the operator $\Gamma_p^\sigma$ can be expressed as 
    \begin{equation*}
  \Gamma_{p}^{\sigma}(f,g)(x):=\frac{1}{2}\left(\mathscr{L}^\sigma_{p,f}(fg)(x)-g(x)L_{p,f}^{\sigma}f(x)-f(x)L_{p,f}^{\sigma}g(x)\right),
\end{equation*} 
\end{rem}

%\begin{equation*}
%  \Gamma_{p,2}^{\sigma}(f,g)(x):=\frac{1}{2}\left\{L_{p,f}\Gamma_{p}^{\sigma}(f,g)(x)-\Gamma_{p}^{\sigma}(g,L_{p,f}^{\sigma}f)(x)-\Gamma_{p}^{\sigma}(f,L_{p,f}^{\sigma}g)(x)\right\}.
%\end{equation*} 

  \begin{defn}\label{def:Cdp}
   We say a signed graph $(G,\sigma)$ satisfies the curvature-dimension inequality $CD_{p}^{\sigma}(K,N)$ at $x\in V$ for $p>1$, $K\in \mathbb{R}$ and $N\in(0,\infty]$ if
    \begin{equation}\label{CDp condition}
        \Gamma_{p,2}^{\sigma}(f,f)(x)\geq\frac{1}{N}(\Delta_{p}^{\sigma}f(x))^{2}+K(\Gamma^{\sigma}_{p}(f,f)(x))^{\frac{2p-2}{p}}
    \end{equation}
    holds for any $f:V\rightarrow\mathbb{R}$ such that $\sigma_{xy}f(y)-f(x)\neq0$ for every neighbor $y\sim x$.
    We say the signed graph $(G,\sigma)$ satisfies the curvature dimension inequality $CD_{p}^{\sigma}(K,N)$ if it satisfies $CD_{p}^{\sigma}(K,N)$ at every vertex $x$.
\end{defn}

\begin{rem}
   The curvature-dimension condition $CD_p(K,N),\,p>1$ on a graph $G$, i.e., a signed graph $(G,\sigma)$ with $\sigma\equiv +1$ being the all-positive sign, has been introduced and studied in \cite{Wang 21,Wang 24}.The condition $CD^\sigma_p(K,N)$ with $p=2$ of a signed graph $(G,\sigma)$ was initiated in \cite{Liu 19}. Our definition is a generalization of both notions.  Notice that the exponent $\frac{2p-2}{p}$ of $\Gamma_p^\sigma(f,f)(x)$ is chosen to ensure the inequality (\ref{CDp condition}) is invariant under scaling the function $f$.
\end{rem}
%By directly calculating, we have the following property of the curvature-dimension inequality on signed graph, which is important to illustrate Definition \ref{def:Cdp} is valid.
\begin{prop}(Switching invariance). Let $(G,\sigma)$ be a signed graph and $\tau: V\to \{\pm 1\}$ be a switching function. Then
$(G,\sigma)$ satisfies $CD_p^\sigma(K,N)$ at a vertex $x\in V$ for  $p>1$, $K\in \mathbb{R}$ and $N\in (0,\infty]$ if and only if the signed graph $(G,\sigma^\tau)$ satisfies $CD_p^{\sigma^\tau}(K,N)$ at $x$.
\end{prop}   
\begin{proof}
Let $f: V\to \mathbb{R}$ be a function. Recall that $\sigma^\tau_{xy}=\tau(x)\sigma_{xy}\tau(y)$. Then we observe that 
$\sigma_{xy}f(y)-f(x)\neq 0$ for any $y\sim x$ if and only if $\sigma^\tau_{xy}(\tau f)(y)-(\tau f)(x)\neq 0$ for any $y\sim x$. Moreover, we check 
\begin{equation*}
    \Gamma^{\sigma^{\tau}}_{p}(\tau f,\tau f)(x)=\Gamma^{\sigma}_{p}(f, f)(x), \,\,\Delta_p^{\sigma^\tau}(\tau f)(x)=\tau(x)\Delta_p^{\sigma}f(x),
\end{equation*}
and therefore
\begin{equation*}
    \Gamma^{\sigma^{\tau}}_{p,2}(\tau f,\tau f)(x)=\Gamma^{\sigma}_{p,2}(f, f)(x).
\end{equation*}
Then, the proposition follows directly by definition.
       % Let $f$ be a function on signed graph $(G,\sigma)$. Let $\tau:V\rightarrow\{+1,-1\}$ be a switching function. Then the curvature-dimension inequality (\ref{CDp condition}) holds on signed graph $(G,\sigma)$ is equivalent to
   % \begin{equation*}
   %     \Gamma_{p,2}^{\sigma^{\tau}}(\tau f,\tau f)(x)\geq\frac{1}{N}(\Delta_{p}^{\sigma^{\tau}}(\tau f(x)))^{2}+K(\Gamma^{\sigma^{\tau}}_{p}(\tau f,\tau f)(x))^{\frac{2p-2}{p}}
  %  \end{equation*}
   % holds on signed graph $(G,\sigma^{\tau})$.
\end{proof}
\begin{prop}
    For $p\geq2$, a signed graph $(G,\sigma)$ satisfies the curvature-dimension inequality $CD_{p}^{\sigma}(K,N)$ at $x\in V$ for $K\in \mathbb{R}$ and $N\in(0,\infty]$ if and only if the curvature-dimension inequality (\ref{CDp condition}) holds for any $f:V\rightarrow\mathbb{R}$.
\end{prop}
\begin{proof}
We only need show that  $CD_{p}^{\sigma}(K,N)$ at $x\in V$ implies (\ref{CDp condition}) for a function $f:V\to \mathbb{R}$ such that there exists $y\sim x$ such that $\sigma_{xy}f(y)-f(x)=0$. We perturb it to be $f_\varepsilon:=f+\varepsilon\mathbf{1}_x$, where $\mathbf{1}_x(y)=1$ if $y=x$ and $0$ otherwise. For any sufficiently small $\varepsilon>0$, we have $\sigma_{xy}f_\varepsilon(y)-f(x)\neq 0$ for any $y\sim x$. Then 
     $CD_{p}^{\sigma}(K,N)$ at $x\in V$ implies that
         \begin{equation*}
        \Gamma_{p,2}^{\sigma}(f_\varepsilon,f_\varepsilon)(x)\geq\frac{1}{N}(\Delta_{p}^{\sigma}f_\varepsilon(x))^{2}+K(\Gamma^{\sigma}_{p}(f_\varepsilon,f_\varepsilon)(x))^{\frac{2p-2}{p}}.
    \end{equation*}
    Letting $\epsilon \to 0$ yields (\ref{CDp condition}) for the function $f$.
\end{proof}
%\begin{rem}
    %Notice that (\ref{CDp condition}) trivially holds at $x\in V$ with $\sigma_{xy}f(y)-f(x)=0$
%\end{rem}

%To prove Lichnerowicz-type estimate for signed $p$-Laplacian, we need to derive the following Lemma.
%\begin{lem}\label{lemma}
  %Suppose that $G$ is a locally finite signed graph and $L^{\sigma}$ is the linearized operator on $G$ defined in (\ref{linearized operator}), we have
  %\begin{equation*}
    %\sum_{x\in V}f(x)L^{\sigma}(|f(x)|^{p-2}f(x))d_{x}=\sum_{x\in V}|f(x)|^{p-2}f(x)L^{\sigma}f(x)d_{x}.
  %\end{equation*}
%\end{lem} 

%\begin{proof}
 % By straightly calculation, 
  %\begin{equation*}
   %\begin{split}
      %& \sum_{x\in V}f(x)L^{\sigma}(|f(x)|^{p-2}f(x))d_{x}\\ 
      %&= \sum_{x\in V}f(x)\sum_{y\sim x}|\sigma_{xy}f(y)-f(x)|^{p-2}(\sigma_{xy}|f|^{p-2}(y)f(y)-|f|^{p-2}(x)f(x))\\
         %& =\sum_{x\sim y}(f(x)-\sigma_{xy}f(y))|\sigma_{xy}f(y)-f(x)|^{p-2}(\sigma_{xy}|f|^{p-2}(y)f(y)-|f|^{p-2}(x)f(x)),
    %\end{split}
  %\end{equation*}
  %and
 % \begin{equation*}
  %\begin{split}
   %&\sum_{x\in V}|f(x)|^{p-2}f(x)L^{\sigma}f(x)d_{x} \\
  % &=\sum_{x\in V}|f(x)|^{p-2}f(x)\sum_{y\sim x}|\sigma_{xy}f(y)-f(x)|^{p-2}(\sigma_{xy}f(y)-f(x))\\
   %&=\sum_{x\sim y}(f(x)-\sigma_{xy}f(y))|\sigma_{xy}f(y)-f(x)|^{p-2}(\sigma_{xy}|f|^{p-2}(y)f(y)-|f|^{p-2}(x)f(x)).
   %\end{split}
  %\end{equation*}
  %Then, the lemma is proved.
%\end{proof}

Now we are prepared to present our Lichnerowicz type eigenvalue estimate for the $p$-Laplacian $\Delta_p^\sigma$ of a signed graph.
\begin{thm}\label{thm:Lich}
   Let $(G,\sigma)$ be a finite connected signed graph satisfying $CD_{p}^{\sigma}(K,N)$ with $p\geq2$, $N>1$ and $K>0$. Then, we have for the first nonzero eigenvalue $\lambda_p^\sigma$ that
   \begin{equation*}
     \lambda_{p}^{\sigma}\geq\left(\frac{1}{\mathrm{vol}(G)}\right)^{\frac{p-2}{2}}\left(\frac{NK}{N-1}\right)^{\frac{p}{2}}.
   \end{equation*}
\end{thm}
\begin{proof}
  Let $f$ be an eigenfunction of the $p$-Laplacian $\Delta_p^\sigma$ with respect to the first nonzero eigenvalue $\lambda_{p}^{\sigma}$. Then, we calculate
  \begin{align*}
       &\sum_{x\in V}\Gamma_{2,p}^{\sigma}(f,f)(x)d_{x}\\
       =&\sum_{x\in V}\left(\frac{1}{2}\mathscr{L}_{p,f}^{\sigma}(\Gamma_{p}^{\sigma}(f,f))(x)-\Gamma_{p}^{\sigma}(f,\Delta_{p}^{\sigma}f)(x)\right)d_{x} \\
          =&\lambda_{p}^{\sigma}\sum_{x\in V}\Gamma_{p}^{\sigma}(f,|f|^{p-2}f)(x)d_{x}\\
         =&\frac{1}{2}\lambda_{p}^{\sigma}\sum_{x\in V}\sum_{y\sim x}|\sigma_{xy}f(y)-f(x)|^{p-2}\left(\sigma_{xy}f(y)-f(x)\right)\left(\sigma_{xy}|f(y)|^{p-2}f(y)-|f(x)|^{p-2}f(x)\right)\\
         =& \frac{1}{2}\lambda_{p}^{\sigma}\sum_{x\in V}\sum_{y\sim x}|\sigma_{xy}f(y)-f(x)|^{p-2}\left(|f(y)|^p-\sigma_{xy}|f(x)|^{p-2}f(x)f(y)\right.\\
         &\hspace{8.5cm}\left.-\sigma_{xy}|f(y)|^{p-2}f(y)f(x)+|f(x)|^p\right)\\
        % \\&\bigg[\left(|f|^{p-2}f^{2}(y)-|f|^{p-2}f^{2}(x)\right)\\
       %  &\quad\quad\quad\quad\quad\quad\quad\quad\quad\quad-f(x)\left(\sigma_{xy}(|f|^{p-2}f)(y)-(|f|^{p-2}f)(x)\right)\\
         %&\quad\quad\quad\quad\quad\quad\quad\quad\quad\quad\quad\quad\quad-(|f|^{p-2}f)(x)\left(\sigma_{xy}f(y)-f(x)\right)\bigg]\\
         =& -\lambda_{p}^{\sigma}\sum_{x\in V}|f(x)|^{p-2}f(x)\sum_{y\sim x}|\sigma_{xy}f(y)-f(x)|^{p-2}(\sigma_{xy}f(y)-f(x))\\
         =& -\lambda_{p}^{\sigma}\sum_{x\in V}|f(x)|^{p-2}f(x)\Delta_{p}^{\sigma}f(x)d_{x}\\
         =& (\lambda_{p}^{\sigma})^{2}\sum_{x\in V}|f(x)|^{2p-2}d_{x}.
  \end{align*}
  Applying curvature-dimension inequality $CD_p^\sigma(K,N)$, we derive
  \begin{equation*}
        (\lambda_{p}^{\sigma})^{2}\sum_{x\in V}|f(x)|^{2p-2}d_{x}\geq
        %\sum_{x\in V}\frac{1}{N}(\Delta_{p}^{\sigma}f(x))^{2}d_{x}+\sum_{x}K(\Gamma^{\sigma}_{p}(f,f)(x))^{\frac{2p-2}{p}}d_{x} \\
         \frac{(\lambda_{p}^{\sigma})^{2}}{N}\sum_{x\in V}|f(x)|^{2p-2}d_{x}+K\sum_{x\in V}(\Gamma_{p}^{\sigma}(f,f)(x))^{\frac{2p-2}{p}}d_{x}.
  \end{equation*}
Rearranging the terms yields
 \begin{equation}\label{CD}
  \frac{N-1}{NK}(\lambda_{p}^{\sigma})^{2}\geq\frac{\sum_{x}\left(\frac{1}{2d_{x}}\sum_{y\sim x}|\sigma_{xy}f(y)-f(x)|^{p}\right)^{\frac{2p-2}{p}}d_{x}}{\sum_{x}|f(x)|^{2p-2}d_{x}}.
 \end{equation} 
 On the other hand, we apply H\"older inequality when $p\geq 2$ to obtain
 % Considering the right-side of inequality (\ref{CD}) and using the H\"older inequality, for $p\geq2$, thus
  \begin{equation}\label{inequality:eigenvalue}
    \begin{split}
       2\lambda_{p}^{\sigma} &=\frac{\sum_{x}\sum_{y\sim x}|\sigma_{xy}f(y)-f(x)|^{p}}{\sum_{x}|f(x)|^{p}d_{x}} \\
         & \leq \frac{\left(\sum_{x}\left(\frac{1}{d_{x}}\sum_{y\sim x}|\sigma_{xy}f(y)-f(x)|^{p}\right)^{\frac{2p-2}{p}}d_{x}\right)^{\frac{p}{2p-2}}\mathrm{vol}(G)^{\frac{p-2}{2p-2}}}{\sum_{x}|f(x)|^{p}d_{x}}\\
         & \leq\left(\frac{\sum_{x}\left(\frac{1}{d_{x}}\sum_{y\sim x}|\sigma_{xy}f(y)-f(x)|^{p}\right)^{\frac{2p-2}{p}}d_{x}}{\sum_{x}|f(x)|^{2p-2}d_{x}}\right)^{\frac{p}{2p-2}} \mathrm{vol}(G)^{\frac{p-2}{2p-2}}.
    \end{split}
  \end{equation}
  In the last inequality above, we use the following $p$-norm inequality
  \begin{equation}\label{ineq:p-norm}   
  \left(\sum^{n}_{i=1}|a_{i}|^{p}d_i\right)^{\frac{1}{p}}\leq\left(\sum^{n}_{i=1}|a_{i}|^{q}d_i\right)^{\frac{1}{q}}
  \end{equation}
for $p\geq q>0$ and $d_{i}\geq1, 1\leq i\leq n$. Indeed, we can assume without loss of generality that $ \left(\sum^{n}_{i=1}|a_{i}|^{p}d_i\right)^{\frac{1}{p}}=1$. Since $d_i\geq 1$, we have $|a_i|\leq 1$ for each $1\leq i\leq n$. Therefore, $|a_i|^p\leq |a_i|^q$ for $p\geq q>0$. This leads to 
$1=\sum^{n}_{i=1}|a_{i}|^{p}d_i\leq \sum^{n}_{i=1}|a_{i}|^{q}d_i$, and hence $1=\left(\sum^{n}_{i=1}|a_{i}|^{p}d_i\right)^{\frac{1}{p}}\leq \left(\sum^{n}_{i=1}|a_{i}|^{q}d_i\right)^{\frac{1}{q}}$.
%the inequality (\ref{ineq:p-norm}) can be proved inductively.
%For $n=1$, we check that 
%  \[|a_{1}| d_{1}^{\frac{1}{p}}\leq|a_{1}| d_{1}^{\frac{1}{q}}\]
%  is true for $p\geq q>0$ and $d_1\geq 1$. Assume that we have proven the inequality (\ref{ineq:p-norm}) for $n-1$, that is, 
%\[\sum_{i=1}^{n-1}|a_i|^p d_{i}\leq\left(\sum_{i=1}^{n-1}|a_i|^q d_{i}\right)^{\frac{p}{q}}.\]
%We need show that
%\[\left(\sum_{i=1}^{n-1}|a_i|^p %d_{i}\right)+d_n|a_n|^p\leq\left(\sum_{i=1}^{n-1}|a_i|^q d_{i}+d_n|a_n|^q\right)^{\frac{p}{q}}.\]
% Without loss of generality, we assume $d_n\cdot |a_n|^p=1$. \textbf{This proof does not seems work.}
%So we need to show that if $f(x)=x^{\frac{p}{q}}$, then $f(x+1)\geq f(x)+1$. But this is clear, as $p\geq q>0$ and $x\geq0$.

Putting (\ref{CD}) and (\ref{inequality:eigenvalue}) together yields
  \begin{equation*}%\label{Holder}
    (2\lambda_{p}^{\sigma})^{\frac{2p-2}{p}}\cdot \mathrm{vol}(G)^{\frac{2-p}{p}}\leq\frac{\sum_{x}\left(\frac{1}{d_{x}}\sum_{y\sim x}|\sigma_{xy}f(y)-f(x)|^{p}\right)^{\frac{2p-2}{p}}d_{x}}{\sum_{x}|f(x)|^{2p-2}d_{x}}\leq 2^{\frac{2p-2}{p}}\frac{N-1}{NK}(\lambda_p^\sigma)^2.
  \end{equation*}
This completes the proof.
\end{proof} 
We conclude with a remark explaining that our result covers the Lichnerowicz type estimate for graph $p$-Laplacian by Wang and Huang \cite{Wang 21}.
\begin{rem} We remark that the curvature term in (\ref{CDp condition}) is different from the curvature term in \cite[Definition 2.4 and (3.2)]{Wang 21}. For a graph $G$ (with all-positive sign) satisfying a slightly different curvature dimension condition $CD_p(K_{WH},N)$ with $p\geq 2$, Wang and Huang show the first nonzero eigenvalue $\lambda_p$ satisfies 
   \begin{equation}\label{eq:HW}
     \lambda_{p}\geq\left(\frac{2}{\mathrm{vol}(G)}\right)^{\frac{p-2}{2}}\left(\frac{NK_{WH}}{N-1}\right)^{\frac{p}{2}},
   \end{equation}
   where $CD_p(K_{WH},N)$ is defined as in Definition \ref{def:Cdp} (in case $\sigma$ is all-positive) via replacing (\ref{CDp condition}) by
   \[        \Gamma_{p,2}(f,f)(x)\geq\frac{1}{N}(\Delta_{p}f(x))^{2}+\frac{K_{WH}}{2}|\nabla f|^{2p-2}(x).\]
    In the above, the curvature term reads as 
   \[\frac{K_{WH}}{2}|\nabla f|^{2p-2}(x):=\frac{K_{WH}}{2d_{x}}\sum_{y\sim x}|\sigma_{xy}f(y)-f(x)|^{2p-2}.\]
   For $p\geq 2$, we derive by H\"older inequality for $p\geq 2$ that
   \[\left(\frac{1}{d_{x}}\sum_{y\sim x}|\sigma_{xy}f(y)-f(x)|^{p}\right)^{\frac{1}{p}}\leq\left(\frac{1}{d_{x}}\sum_{y\sim x}|\sigma_{xy}f(y)-f(x)|^{2p-2}\right)^{\frac{1}{2p-2}}.\]
 Hence, we obtain
 \[\frac{K_{WH}}{2}|\nabla f|^{2p-2}(x)\geq 2^{\frac{p-2}{p}}K_{WH}(\Gamma_p(f,f)(x))^{\frac{2p-2}{p}}.\]
 That is, Wang and Huang's $CD_p(K_{WH},N)$ with $p\geq 2$ implies our $CD_p^\sigma(2^{\frac{p-2}{p}}K_{WH},N)$ with $p\geq 2$ and all-positive $\sigma$. Therefore, Wang and Huang's estimate (\ref{eq:HW}) is a consequence of Theorem \ref{thm:Lich}.
 %$(\Gamma^{\sigma}_{p}(f))^{\frac{2p-2}{p}}\leq|\nabla^{\sigma}f|^{2p-2}$.  And we omit coefficient $\frac{1}{2}$ in our definition for the reason that we want to get the Lichnerowicz estimates without coefficient $2^{\frac{p-2}{2}}$.
\end{rem}

\vspace{.3cm}

{\bf{Acknowledgements:}} This work is supported by the National Key R and D Program of China 2020YFA0713100 and the National Natural Science Foundation of China No. 12031017. We are very grateful to Chuanyuan Ge, Chunyang Hu and Wenbo Li for their inspiring discussions, especially on various explicit constructions of signed graphs with particular spectral and curvature properties.

\appendix
\section{Curvature of hypercube graphs with one negative edge}\label{section:Appendix}

In this Appendix, we calculate the curvature of a hypercube graph $Q^n$ with a sign $\sigma$ for which $(Q^n,\sigma)$ has precisely one negative edge.

We first recall the curvature matrix method for calculating the curvature of a general signed graph $(G,\sigma)$ developed in \cite{HL22}, which is a nontrivial generalization of results for unsigned graphs in \cite{CKLP22,CLP20,Sic20,Sic21}. The method works for graphs with general vertex and edge weights and general orthogonal or unitary group signs. For the readers' convenience, we survey the method briefly in the setting of the current paper. 

Let $(G,\sigma)$ be a signed graph. We define the transition rate $p_{xy}$ from $x$ to $y$ as  \[p_{xy}:=\begin{cases}
\frac{1}{d_x} & \text{if $x\sim y$,} \\
0& \text{otherwise.}
\end{cases}\] Next, we consider the curvature at a vertex $x$ of $(G,\sigma)$.
Let us denote the neighbors of $x$ by $y_1,\ldots, y_m$ with $m=d_x$, and the vertices with distance $2$ to $x$ by $z_1,\ldots, z_\ell$. We define an $(m+1)\times (m+1)$ symmetric matrix $Q(x)$, whose rows and columns are indexed by $x,y_1,\ldots, y_m$, as below (see \cite[(A.10)-(A.13)]{HL22}):
\begin{align}
&4Q(x)_{x,x}=3\sum_{i=1}^mp_{xy_{i}}p_{y_{i}x}+1-\sum_{k=1}^\ell\frac{(\sum_{i=1}^mp_{xy_{i}}p_{y_{i}z_{k}}\sigma_{xy_{i}}\sigma_{y_{i}z_{k}})^2}{\sum_{i=1}^mp_{xy_{i}}p_{y_{i}z_{k}}},\label{eq:Q10}\\
&4Q(x)_{x,y_i}=\sum_{j=1,j\neq i}^mp_{xy_j}p_{y_jy_i}\sigma_{xy_j}\sigma_{y_jy_i}-2\left(p_{y_ix}+1\right)p_{xy_i}\sigma_{xy_i}\notag\\
&\hspace{3.5cm}+2\sum_{k=1}^\ell\frac{(\sum_{j=1}^mp_{xy_{j}}p_{y_{j}z_{k}}\sigma_{xy_{j}}\sigma_{y_{j}z_{k}})p_{xy_{i}}p_{y_{i}z_{k}}\sigma_{y_{i}z_{k}}}{\sum_{j=1}^mp_{xy_{j}}p_{y_{j}z_{k}}},\,\forall\,i,\\
&4Q(x)_{y_{i},y_{i}}=\sum_{j=1,j\neq i}^m p_{xy_j}p_{y_jy_{i}}+2\left(p_{xy_i}+1\right)p_{xy_i}-4\sum_{k=1}^\ell\frac{p_{xy_{i}}^2p_{y_{i}z_{k}}^2}{\sum_{j=1}^mp_{xy_{j}}p_{y_{j}z_{k}}},\,\forall \,i,\\
&4Q(x)_{y_{i},y_j}=-2(p_{xy_{i}}p_{y_iy_{j}}+p_{xy_j}p_{y_jy_i})\sigma_{y_{i}y_{j}}+2p_{xy_{i}}p_{xy_{j}}\sigma_{xy_{i}}\sigma_{xy_{j}}\notag\\
&\hspace{5cm}-4\sum_{k=1}^\ell\frac{p_{xy_{i}}p_{y_{i}z_{k}}p_{xy_j}p_{y_jz_{k}}\sigma_{y_{i}z_{k}}\sigma_{y_jz_{k}}}{\sum_{h=1}^mp_{xy_{h}}p_{y_{h}z_{k}}},\,\forall\,i,\,\forall\,j\neq i.\label{eq:Q13}
\end{align}
We define another $(m+1)\times (m+1)$ matrix $B_0$ as (see \cite[(3.33)]{HL22})
\begin{equation}
B_0:=
\begin{pmatrix}\label{eq:canonicalB}
    1 & \sigma_{xy_1}                            & \cdots  & \sigma_{xy_m} \\
          & \frac{1}{\sqrt{p_{xy_{1}}}} &        &       \\
          &                                  & \ddots &       \\
          &                                  &        & \frac{1}{\sqrt{p_{xy_{m}}}}
\end{pmatrix},
\end{equation}
and denote the product $2B_0Q(x)B_0^\top$ by a block matrix as follows:
\begin{equation}\label{eq:BQmatrix}
    2B_0Q(x)B_0^\top=\begin{pmatrix}
        a & \omega^\top \\
        \omega & (2BQ(x)B^\top)_{\hat{1}}
    \end{pmatrix}.
\end{equation}
Here the notation $(2BQ(x)B^\top)_{\hat{1}}$ stands for the $m\times m$ matrix obtained from $2B_0Q(x)B_0^\top$ by removing the first column and the first row. Notice that we always have $a\geq 0$ \cite[Proposition A.1]{HL22}. The curvature matrix at $x$ in $(G,\sigma)$ is then defined to be \cite[Definition 3.1]{HL22}
\begin{equation}\label{eq:curvatuer_matrix}
    A_\infty(G,\sigma, x):=(2BQ(x)B^\top)_{\hat{1}}-\omega a^{-1}\omega^\top,
\end{equation}
where we use the convention $0^{-1}=0$. Moreover, we define an $m$-dimensional vector $\mathbf{v}_0$ as
\begin{equation}\label{eq:canonicalv0}
\mathbf{v}_0:=\begin{pmatrix}
  \sqrt{p_{xy_1}}\sigma_{xy_1} \\\vdots\\ \sqrt{p_{xy_m}}\sigma_{xy_m}
\end{pmatrix},
\end{equation}
and the $N$-dimensional curvature matrix as 
\begin{equation}
       A_N(G,\sigma, x):= A_\infty(G,\sigma, x)-\frac{1}{N}\mathbf{v}_0\mathbf{v}_0^\top.
\end{equation}

Now we are prepared to formulate the curvature at a vertex as the smallest eigenvalue of the corresponding curvature matrix. 
\begin{thm}[{\cite[Theorem 3.2]{HL22}}]\label{thm:HL}
    Let $x$ be a vertex in a signed graph $(G,\sigma)$. For any $N\in (0,\infty]$, we define $K_{G,\sigma,x}(N)$ to be the largest $K$ such that $CD^\sigma(K,N)$ holds at $x$. Then we have
    \[K_{G,\sigma,x}(N)=\lambda_{\min}(A_N(G,\sigma, x)),\]
where $\lambda_{\min}(\cdot)$ stands for the minimal eigenvalue of the matrix.
\end{thm}

Applying Theorem \ref{thm:HL}, we show the following result on the curvature of a signed hypercube graph.

\begin{prop}
   Let $(Q^n,\sigma)$ be the $n$-dimensional hypercube graph for some $n\in \mathbb{N}$ with the sign $\sigma$ which has precisely one negative edge. Then  the signed hypercube graph $(Q^n,\sigma)$ satisfies $CD^\sigma(0,\infty)$ if and only if $n=1,2$. 
\end{prop}
\begin{proof}
Let $x$ be a vertex in $(Q^n,\sigma)$ with neighbors $y_1,\ldots,y_n$ such that 
    \[\sigma_{xy_1}=-1, \sigma_{xy_2}=\cdots=\sigma_{xy_n}=+1.\]
By definition of hypercube graphs, every vertex $z_k$ with $d(x,z_k)=2$, if exists,  has exactly two neighbors in $\{y_1,\ldots,y_n\}$. Hence, we have
\[\sum_{j=1}^np_{xy_j}p_{y_jz_k}=\frac{2}{n^2}, \,\,\text{and}\,\,\sum_{j=1}^mp_{xy_j}p_{y_jz_k}\sigma_{xy_j}\sigma_{y_jz_k}=\begin{cases}
0 & \text{if $z_k\sim y_1$,} \\
\frac{2}{n^2}& \text{otherwise.}
\end{cases}\]
Furthermore, we derive from (\ref{eq:Q10})-(\ref{eq:Q13}) that
\begin{equation*}
    4Q(x)=\frac{2}{n^2}\begin{pmatrix}
    3n-1 & n+1 & -3& \cdots  & -3\\
       n+1   & 2 &   -2  & \cdots   & -2      \\
        -3  &  -2 & 2 & \cdots &   0    \\
        \vdots & \vdots & \vdots & \ddots & \vdots\\
          -3& -2& 0  & \cdots       & 2
\end{pmatrix}.
\end{equation*}
By (\ref{eq:canonicalB}), we have
\begin{equation*}
B_0:=
\begin{pmatrix}
    1 & -1 & 1  & \cdots  & 1 \\
      0    & \sqrt{n} & 0 & \cdots & 0          \\
      0 & 0 & \sqrt{n} &\cdots & 0\\
      \vdots    &  \vdots & \vdots                                 & \ddots &     \vdots  \\
       0   &     0 & 0                             &   \cdots     & \sqrt{n}
\end{pmatrix}.
\end{equation*}
A direct calculation yields
\begin{equation*}
    2B_0Q(x)B_0^\top=\frac{1}{n^2}\begin{pmatrix}
    n-1 & -(n-1)\sqrt{n} & \sqrt{n}& \cdots  & \sqrt{n}\\
       -(n-1)\sqrt{n}  & 2n &   -2n  & \cdots   & -2n    \\
        \sqrt{n}  &  -2n & 2n & \cdots &   0    \\
        \vdots & \vdots & \vdots & \ddots & \vdots\\
         \sqrt{n} & -2n& 0  & \cdots       & 2n
\end{pmatrix}.
\end{equation*}
When $n=1$, we have $a:=\frac{n-1}{n^2}=0$, and hence derive from (\ref{eq:curvatuer_matrix}) the $1\times 1$ curvature matrix
$$  A_{\infty}(Q^1,\sigma, x)=\begin{pmatrix}
    2
\end{pmatrix}.$$
Applying Theorem \ref{thm:HL}, we have 
$K_{Q^1,\sigma,x}=2$.
That is, the signed graph $(Q^1,\sigma)$ satisfies $CD^\sigma(2,\infty)$.

Noticing that $a:=\frac{n-1}{n^2}>0$ for any $n\geq 2$, we obtain by (\ref{eq:curvatuer_matrix}) the $n\times n$ curvature matrix as
\begin{align*}
    A_{\infty}(Q^n,\sigma, x)=&\frac{1}{n}\begin{pmatrix}
       2 &   -2  & \cdots   & -2    \\
         -2 & 2 & \cdots &   0    \\
      \vdots & \vdots & \ddots & \vdots\\
          -2& 0  & \cdots       & 2
\end{pmatrix}-\omega a^{-1}\omega^\top,
\end{align*}
with $\omega^\top:=\frac{1}{n^2}\begin{pmatrix}-(n-1)\sqrt{n} & \sqrt{n}& \cdots  & \sqrt{n}\end{pmatrix}$. Therefore, we conclude that 
\begin{align*}
%&-\frac{1}{n^4}\begin{pmatrix}
%       -(n-1)\sqrt{n}  \\
%        \sqrt{n}  \\
%        \vdots \\
%         \sqrt{n}
%\end{pmatrix}\cdot \frac{n^2}{n-1}\cdot\begin{pmatrix}
%-(n-1)\sqrt{n} & \sqrt{n}& \cdots  & \sqrt{n}
%\end{pmatrix}\\
  A_{\infty}(Q^n,\sigma, x)= \frac{1}{n}\begin{pmatrix}
    3-n & -\mathbf{1}_{n-1}^\top \\
        -\mathbf{1}_{n-1} & 2I_{n-1}-\frac{1}{n-1}J_{n-1}
\end{pmatrix},
\end{align*}
where $I_{n-1}$ is the $(n-1)\times (n-1)$ identity matrix, $J_{n-1}$ is the $(n-1)\times (n-1)$ all-one matrix, and $\mathbf{1}_{n-1}$ is the all-one $(n-1)$-dimensional vector. The eigenvalues of $A_{\infty}(Q^n,\sigma, x)$ are $\frac{2-n}{n}$ with multiplicity $1$ and $\frac{2}{n}$ with multiplicity $n-1$. Therefore, Theorem \ref{thm:HL} tells that
\[K_{Q^n,\sigma,x}=\frac{2-n}{n}, \,\,\text{for}\,\,n\geq 2.\]
When $n=2$, we have $K_{Q^2,\sigma,x}=0$. By switching invariance, this is true for any vertex in $(Q^2,\sigma)$. Therefore, $(Q^2,\sigma)$ satisfies $CD^\sigma(0,\infty)$. 
In contrast, $(Q^n,\sigma),\,n\geq 3$ does not satisfy $CD^\sigma(0,\infty)$ since $\frac{2-n}{n}<0$.
\end{proof}

\end{document}